\documentclass[12pt]{amsart}
\usepackage{amsfonts}
\usepackage{amsmath}
\usepackage{amssymb}
\usepackage[margin=1in]{geometry}
\usepackage{tikz}
\usepackage{amsthm}    
\usepackage{mathtools}							
\usepackage{commath}							
\usepackage{enumitem}  							

\usepackage[english]{babel}						

\makeatletter
\@namedef{subjclassname@2020}{\textup{2020} Mathematics Subject Classification}
\makeatother

\newtheorem{theorem}{Theorem}[section]

\newtheorem{lemma}[theorem]{Lemma}

\theoremstyle{definition}

\newtheorem{problem}[theorem]{Problem}

\numberwithin{equation}{section}

\newcommand{\expm}{\exp^{\ast}}
\newcommand{\nconv}{^{\ast n}}

\newcommand{\N}{\mathbb{N}}
\newcommand{\Z}{\mathbb{Z}}

\newcommand{\C}{\mathbb{C}}

\renewcommand{\Re}{\operatorname{Re}}
\renewcommand{\Im}{\operatorname{Im}}

\newcommand{\I}{\mathrm{i}}
\newcommand{\e}{\mathrm{e}}

\newcommand{\eps}{\varepsilon}
\newcommand{\vphi}{\varphi}

\newcommand{\MP}{\mathcal{P}}

\newcommand{\LHS}{\text{LHS}}
\newcommand{\RHS}{\text{RHS}}

\newcommand{\GHL}{\Gamma_{\text{H-L}}}
\newcommand{\mmax}{m_{\text{max}}}

\DeclareMathOperator{\Li}{Li}

\DeclareMathOperator{\res}{Res}
\DeclareMathOperator{\sgn}{sgn}

\begin{document}

\title{Beurling integers with RH and large oscillation}

\author[F. Broucke]{Frederik Broucke}
\thanks{F. Broucke was supported by the Ghent University BOF-grant 01J04017}

\address{Department of Mathematics: Analysis, Logic and Discrete Mathematics\\ Ghent University\\ Krijgslaan 281\\ 9000 Gent\\ Belgium}
\email{fabrouck.broucke@ugent.be}
\email{gregory.debruyne@ugent.be}
\email{jasson.vindas@ugent.be}

\author[G.~Debruyne]{Gregory Debruyne}
\thanks{G.~Debruyne acknowledges support by Postdoctoral Research Fellowships of the Research Foundation--Flanders (grant number 3E006818) and the Belgian American Educational Foundation. The latter one allowed him to do part of this research at the University of Illinois at Urbana-Champaign.}

\author[J. Vindas]{Jasson Vindas}
\thanks {J. Vindas was partly supported by Ghent University through the BOF-grant 01J04017 and by the Research Foundation--Flanders through the FWO-grant 1510119N}


\begin{abstract}
We construct a Beurling generalized number system satisfying the Riemann hypothesis and whose integer counting function displays extremal oscillation in the following sense.  The prime counting function of this number system satisfies $\pi(x)= \Li (x)+ O(\sqrt{x})$, while its integer counting function satisfies the oscillation estimate $N(x) = \rho x + \Omega_{\pm}\bigl(x\exp(-c\sqrt{\log x\log\log x})\bigr)$ for some $c>0$, where $\rho>0$ is its asymptotic density. The construction is inspired by a classical example of H. Bohr for optimality of the convexity bound for Dirichlet series, and combines saddle-point analysis with the Diamond-Montgomery-Vorhauer probabilistic method via random prime number system approximations.
\end{abstract}

\maketitle

\section{Introduction}

In \cite{DiamondMontgomeryVorhauer}, H.~G.~Diamond, H.~L.~ Montgomery, and U.~M.~A.~Vorhauer constructed a Beurling generalized number system with very regularly distributed integers, but whose distribution of prime numbers has large oscillation. In fact, given $1/2<\theta<1$, they showed the existence of Beurling numbers with integer counting function $N$  satisfying the asymptotic estimate
\begin{equation}
\label{eq: N regular}
N(x)=\rho x + O(x^{\theta}),
\end{equation}
for some $\rho>0$, and whose primes realize the de la Vall\'{e}e Poussin error term in the Prime Number Theorem (PNT), namely, with prime counting function satisfying the oscillation estimate
\begin{equation*}
\pi(x)=\Li(x)  + \Omega_{\pm}\bigl(x\exp\bigl(-c\sqrt{\log x}\bigr)\bigr)
\end{equation*}
for some $c>0$, where $\Li(x)$ stands for the logarithmic integral and the notation $f(x)=\Omega_{\pm}(g(x))$ means that there is $c'>0$ such that the inequalities $f(x)>c'g(x)$ and $f(x)<-c'g(x)$ hold infinitely often for arbitrary large values of $x$. Their Beurling number system has the additional feature that its associated zeta function, defined as
\begin{equation}
\label{eq: def zeta}
\zeta(s)=\int_{1^{-}}^{\infty} x^{-s}\dif N(x)= \exp\left(\int_{1}^{\infty}x^{-s}\dif \Pi(x)\right)
\end{equation}
with $\Pi(x)= \sum_{j=1}^{\infty} j^{-1}\pi(x^{1/j})$, also realizes the classical de la Vall\'{e}e Poussin zero-free region; in particular, the Riemann hypothesis (RH) fails for it.

The existence of such a number system proves the optimality of E.~Landau's classical PNT \cite{Landau1903}, which, recasted in the language of Beurling numbers, states that \eqref{eq: N regular} always implies the validity of the PNT in the form
\begin{equation}
\label{eq: de la V-P error}
\pi(x)=\Li(x)  + O\bigl(x\exp\bigl(-c\sqrt{\log x}\bigr)\bigr),
\end{equation}
for some $c>0$. That Landau's theorem was optimal is part of a long-standing open problem, posed by P.~T.~Bateman and Diamond in \cite[13B, p.~199]{BD69}. The number system constructed by Diamond, Montgomery, and Vorhauer also provides the valuable information that a zeta function might not have a wider zero-free region than that of de la Vall\'{e}e Poussin if we only require that Beurling's property that the integers have multiplicative structure and \eqref{eq: N regular} hold. Several arguments from \cite{DiamondMontgomeryVorhauer} have been sharpened by W.-B.~Zhang in \cite{Zhang2007}. Interestingly, Zhang complemented these results by showing that there are also Beurling number systems for which, in contrast, the RH and the asymptotic estimate \eqref{eq: N regular} both hold. 

In this work we shall establish the existence of a counterpart of the result from \cite{DiamondMontgomeryVorhauer} for number systems whose integers display large oscillation in the sense of the following theorem, which is in fact our main result. The present article is fully devoted to its proof. Let us first recall that, following A.~Beurling \cite{Beurling1937} (see also \cite{BD69,DiamondZhangbook}), a set of generalized primes is simply an unbounded sequence of real numbers $\mathcal{P}:\:p_{1}\leq p_{2}\leq \dots$ subject to the only requirement $p_{1}>1$. Its associated set of generalized integers is the multiplicative semigroup generated by $\mathcal{P}$ and 1 (where multiplicities according to different representations as products of generalized primes are taken into account). The symbols  $N(x)$ and $\pi(x)$ already used above are the functions that count the number of generalized integers and primes, respectively, not exceeding a given number $x$. 
 \begin{theorem}
\label{th: discrete optimality HL}
There exists a Beurling generalized number system such that 
\begin{equation}
\label{eq: discrete prime approximation}
	\pi (x) = \Li(x) + O(\sqrt{x})
\end{equation}
and, for any constant $c>2\sqrt{2}$, 
\begin{equation}
\label{eq: discrete integer approximation}
 N(x) = \rho x + \Omega_{\pm}\bigl(x\exp(-c\sqrt{\log x\log\log x})\bigr),
\end{equation}
where $\rho>0$ is the asymptotic density of $N$.
\end{theorem}

Theorem \ref{th: discrete optimality HL} proves, apart from the numerical value of the constant $c$, the optimality of the ensuing result due to T.~W.~Hilberdink and M.~L.~Lapidus \cite{HilberdinkLapidus2006} (see also \cite{GarunkstisKaziulyte}), which,  in turn, might be regarded as the analog of Landau's result in the reverse direction, that is, for number systems with very regular prime counting function (namely, satisfying \eqref{eq: regular estimate pi}).

\begin{theorem}[{\cite{HilberdinkLapidus2006}}]
\label{th: HilberdinkLapidus} 
Suppose the generalized prime counting function satisfies 
\begin{equation}
\label{eq: regular estimate pi}
\pi(x) = \Li(x) + O(x^{\theta})
\end{equation}
 for some $0<\theta <1$. Then, there are constants $\rho>0$ and $c>0$ such that its associated generalized integer counting function satisfies  
\begin{equation}
\label{eq: N with exp error term}
	N(x) = \rho x+ O\bigl(x\exp\bigl(-c\sqrt{\log x\log\log x}\bigr)\bigr).
\end{equation}
\end{theorem}

Theorem \ref{th: HilberdinkLapidus} is an improvement to a result of P.~Malliavin \cite{Malliavin}, who obtained the weaker error term $O\bigl(x\exp\bigl(-c\sqrt[3]{\log x}\bigr)\bigr)$ in \eqref{eq: N with exp error term} also under the hypothesis \eqref{eq: regular estimate pi}. Hilberdink and Lapidus have actually proved that \eqref{eq: N with exp error term} holds with the constant $c = \min\{\sqrt{1-\theta}/4, \sqrt{2}/8\}$. We mention that, using a variant of M.~Balazard's technique from \cite{Balazard1999}, one might improve this value to any $c<\min\{\sqrt{2(1-\theta)}, 1\}$. We include details about  how to obtain this improvement in Appendix \ref{appendix}.

Note that \eqref{eq: discrete prime approximation} implies the Beurling number system satisfies the RH, that is, its zeta function 
analytically extends to $\Re s> 1/2$, except for a simple pole located at $s=1$, and has no zeros in this half-plane. In this regard, it is worthwhile to compare our generalized number system from Theorem \ref{th: discrete optimality HL} with earlier examples by E.~Balanzario \cite{Balanzario1998} and Zhang \cite{Zhang2007}. On the one hand, in Balanzario's example\footnote{Balanzario constructs a `continuous' example in \cite{Balanzario1998}, but F.~A.~Al-Maamori  has recently shown \cite{Al-Maamori} via probabilistic arguments  that the example of Balanzario can be discretized.} the oscillation estimate \eqref{eq: discrete integer approximation} holds for $N$, but $\pi$ only satisfies the weaker asymptotic relation \eqref{eq: de la V-P error}. On the other hand, Zhang's example quoted above has generalized prime counting function satisfying \eqref{eq: discrete prime approximation}  (hence the RH holds here), but its generalized integer counting function is too regular for our purposes, i.e., \eqref{eq: N regular} holds for it. Our Beurling number system enjoys the most extremal features of Balanzario's and Zhang's instances and, in turn, neither of their constructions is able to simultaneously deliver  \eqref{eq: discrete prime approximation} and \eqref{eq: discrete integer approximation}. 

Our method for showing Theorem \ref{th: discrete optimality HL} is first to construct a \emph{continuous} analog of a number system having the desired properties and then to find a suitable discrete approximation to it, yielding the sought discrete Beurling number system. As in classical number theory, the key property linking $N$ and $\Pi$ for Beurling numbers is the zeta function identity \eqref{eq: def zeta}, or equivalently, the relation 
$\dif N=\exp^{\ast}\left(\dif \Pi\right)$, where the exponential is taken with respect to the multiplicative convolution of measures \cite{Diamond1970, DiamondZhangbook}. The latter exponential identity then makes sense for not necessarily atomic measures $\dif \Pi$ and $\dif N$ (hereafter supported in $[1,\infty)$ and non-negative), giving rise to the possibility to consider `continuous number systems'. The idea of using such continuous analogs to show optimality of results of course goes back to Beurling's seminal paper \cite{Beurling1937}, and has extensively been exploited by several authors since then\footnote{A large number of instances has been collected in the monograph \cite{DiamondZhangbook}; see also the recent work \cite{D-S-V2016}.}. We shall show the next result.

\begin{theorem}
\label{th: optimality HL}
There exists an absolutely continuous prime distribution function $\Pi_{C}$ with associated integer distribution function $N_{C}$ (determined by $\dif N_{C} = \exp^{\ast}(\dif\Pi_{C})$) such that
\begin{equation}
\label{eq: asymp behavior number system 1} 
	\Pi_{C}(x) = \int_{1}^{x} \frac{1-u^{-1}}{\log u}\dif u +O(1)
		\end{equation}
	and, for  any constant $c>2\sqrt{2}$,
\begin{equation}
\label{eq: asymp behavior number system 2} 
N_{C}(x) = \rho x + \Omega_{\pm}\bigl(x\exp\bigl(-c\sqrt{\log x\log\log x}\bigr)\bigr),
\end{equation}
where $\rho>0$ is the asymptotic density of $N_{C}$.
\end{theorem}

It is worth pointing out that Theorem \ref{th: optimality HL} explains the fundamental difference between Zhang's work \cite{Zhang2007} and ours. Zhang's example of a Beurling number system satisfying the RH can basically be considered as a discretization of the absolutely continuous prime distribution $P(x) =\int_{1}^{x} (1-u^{-1})/\log u \dif u$, which is also the starting point of numerous constructions in the theory of Beurling primes due to fact that it leads to a highly regular continuous number system \cite{DiamondZhangbook}. Indeed, the exponential of $\dif P$ is $\exp^{\ast}\left(\dif P \right)= \delta_1+\dif x$, with $\delta_1$ the Dirac delta concentrated at $1$ and $\dif x$ the Lebesgue measure, so that its associated integer distribution function is simply the function $x$ on $[1,\infty)$, while its associated zeta function is the meromorphic function $s/(s-1)$. In order to demonstrate Theorem \ref{th: optimality HL}, one needs to show the existence of a perturbation of $\dif P$ whose exponential displays the extremal behavior \eqref{eq: asymp behavior number system 2}. It is quite remarkable that such a small perturbation in the  primes can reinforce itself in such a way, analogous to the butterfly effect, to produce such a big discrepancy in the integers. Nevertheless, it cannot fully destroy the integer law.
 
The construction and analysis of a continuous example possessing the properties stated in Theorem \ref{th: optimality HL} that we give is quite involved, and is the subject of Sections \ref{sec: Setup and overview} through \ref{section: conclusion proof lemma}. We refer to Section \ref{sec: Setup and overview} for a sketch of the proof of Theorem \ref{th: optimality HL}, and some insights into the motivation for our considerations.
 
We confine ourselves here to mention that the starting template for our example is an old construction from H.~Bohr's thesis \cite{Bohr1910}, which we shall manipulate to achieve the desired properties. The estimate \eqref{eq: asymp behavior number system 1} will automatically be  satisfied by construction; the challenging part is to match it with the oscillation estimate  \eqref{eq: asymp behavior number system 2}, which we shall actually deduce from a certain extremal behavior of the associated zeta function that we will generate in our construction. As a matter of fact, most of our work in the subsequent sections is a detailed saddle-point analysis of this zeta function. After establishing that our continuous example satisfies all requirements from Theorem \ref{th: optimality HL}, we proceed to carry out a discretization procedure for it in Section \ref{section discretization}. This procedure will be accomplished by adapting to our problem the Diamond-Montgomery-Vorhauer probabilistic scheme \cite{DiamondMontgomeryVorhauer} based upon approximations by random Beurling primes. Our adaptation of this scheme delivers sufficiently strong bounds for the modulus of the relevant zeta functions; however, for our application, we also need to keep good control on the argument of the randomly found zeta function, for which the direct bounds from the Diamond-Montgomery-Vorhauer method appear to be insufficient. We will resolve this issue with a new idea of adding finitely many well-chosen primes to our number system.

We end this introduction by placing Theorem \ref{th: discrete optimality HL} in the context of a long-standing open problem. In fact, the Beurling number system we exhibit in this article and the Diamond-Montgomery-Vorhauer example from \cite{DiamondMontgomeryVorhauer} are just two pieces of a fascinating unsolved puzzle essentially raised by Bateman and Diamond in \cite[13B, p. 199]{BD69} and having its roots in the work of Malliavin. 

In \cite{Malliavin}, Malliavin discovered that the two asymptotic relations 
\begin{equation}
\label{eq: Malliavin PNT}
\tag{P$_{\alpha}$}
\pi(x)= \operatorname*{Li}(x)+ O(x\exp (-c \log^{\alpha} x ))
\end{equation}
and
\begin{equation}
\label{eq: Malliavin density}
\tag{N$_{\beta}$}
N(x)= \rho x+ O(x\exp (-c' \log^{\beta} x )) \qquad (\rho>0),
\end{equation}
for some $c>0$ and $c'>0$, are closely related to each other in the sense that if  \eqref{eq: Malliavin density} holds for a given $0<\beta\leq1$, then (P$_{\alpha^{\ast}}$) is satisfied for a certain $\alpha^\ast$, and vice versa the relation \eqref{eq: Malliavin PNT} for a given $0<\alpha\leq1$ implies that (N$_{\beta^{\ast}}$) holds for a certain $\beta^{\ast}$. Writing $\alpha^{\ast}(\beta)$ and $\beta^{\ast}(\alpha)$ for the best possible\footnote{To be precise, the suprema over all admissible values $\alpha^{\ast}$ and $\beta^{\ast}$ in these implications, respectively.} exponents in these implications, we have:
 
 \begin{problem}\label{Malliavin open problem} Given any $\alpha,\beta\in (0,1]$, find the best exponents $\alpha^{\ast}(\beta)$ and $\beta^{\ast}(\alpha)$. \end{problem}

Theorem \ref{th: discrete optimality HL} and Theorem \ref{th: HilberdinkLapidus} together then yield $\beta^{\ast}(1)=1/2$, while the work of Diamond, Montgomery, and Vorhauer in combination with Landau's PNT gives $\alpha^{\ast}(1)=1/2$. These are the only   two cases where a solution to Problem \ref{Malliavin open problem} is known, and for the remaining values $0<\alpha<1$ and $0<\beta<1$ the question remains wide open. It has been conjectured by Bateman and Diamond that $\alpha^{\ast}(\beta)=\beta/(1+\beta)$. However, the best known admissible value \cite[Theorem~16.8, p.~187]{DiamondZhangbook} when $0<\beta<1$ is $\alpha^{\ast}\approx \beta/(6.91+\beta)$; this falls far short of the conjectural exponent. It is also believed that $\beta^{\ast}(\alpha)=\alpha/(\alpha+1)$ for each $0<\alpha<1$, which is suggested by the work of Diamond, who showed in \cite{Diamond1970} that the hypothesis \eqref{eq: Malliavin PNT} actually ensures a slightly better asymptotic estimate than (N$_{\alpha/(\alpha+1)})$, namely,
$$
N(x)= \rho x+ O(x\exp (-c' (\log x \log \log x )^{\frac{\alpha}{\alpha+1}})
$$
for some $\rho,c'>0$. Al-Maamori \cite{Al-Maamori} has recently found an upper bound, providing $\alpha/(\alpha+1)\leq \beta^{\ast}(\alpha)\leq \alpha$ when $0<\alpha<1$, but we strongly believe that there is still room for improvement here.  In fact, it is worth noting that Theorem \ref{th: discrete optimality HL} leads to $\beta^{\ast}(\alpha)\leq 1/2$, which is better than Al-Maamori's upper bound in the range $1/2< \alpha< 1$.

The authors thank Harold G. Diamond for his useful comments and remarks.
\section{Setup and overview of construction for the continuous example}
\label{sec: Setup and overview}
Before introducing our continuous analog of a Beurling number system that will satisfy the properties stated in Theorem \ref{th: optimality HL}, let us first sketch the motivation for its definition. 

In the complex analysis proof of Theorem \ref{th: HilberdinkLapidus} from \cite{HilberdinkLapidus2006}, the error term in \eqref{eq: N with exp error term} comes from an integral of the zeta function over the contour given by
\[
	s(t) = 1 - \frac{\log\log \abs{t}}{\log \abs{t}} + \I t, \quad \abs{t}\ge3.
\]
In order to generate an example for which this error term is reached, one might thus attempt to find a zeta function having certain extremal growth properties along this contour. Taking a closer look at the proof of Theorem \ref{th: HilberdinkLapidus}, one sees that the bound obtained\footnote{In \cite{HilberdinkLapidus2006}, they actually prove the convexity bound for $-\zeta'/\zeta$, the Mellin-Stieltjes transform of $\psi(x)=\int_{1}^{x}\log u \dif \Pi(u)$, and then derive a bound for $\log\zeta$ on the given contour via integration. One can however start from $\Pi$ instead of $\psi$ and  directly prove the convexity bound for $\log\zeta$.}  in \cite{HilberdinkLapidus2006} for $\log \zeta$ is essentially the convexity bound (cf. \cite[Theorem 1.19 and Theorem 1.20, pp.~201--202]{Tenenbaumbook}) for the Dirichlet series of $\log \zeta$. 
The core of our idea is to construct a prime counting function whose Mellin-Stieltjes transform attains this convexity bound. The inspiration for our construction goes back to Bohr, who showed in his thesis \cite{Bohr1910} via an ingenious example that the convexity bound for Dirichlet series is basically optimal. The example we shall now study is in fact a subtle variant of Bohr's example, modified in such a way that it indeed gives rise to an absolutely continuous prime distribution function having the desired properties to deliver a proof of Theorem \ref{th: optimality HL}.

Let us set up our construction. We consider a positive sequence $(\tau_{k})_{k}$ with $\tau_{0}\ge3$ and rapidly increasing to $\infty$, a positive sequence $(\delta_{k})_{k}$ tending to 0, and a sequence $(\nu_{k})_{k}$ that takes values between 2 and 3. For $x>1$, we then define
\[
	\Pi_{C}(x) \coloneqq \int_{1}^{x}\frac{1-1/u}{\log u}\dif u + \sum_{k=0}^{\infty}R_{k}(x), \quad \text{with} \quad 
	R_{k}(x) \coloneqq 
		\begin{cases}
		\sin(\tau_{k}\log x)	&\mbox{for } \tau_{k}^{1+\delta_{k}} < x \leq \tau_{k}^{\nu_{k}},\\
		0				&\mbox{otherwise}.
		\end{cases}
\]

We make a choice for the sequences $(\tau_{k})_{k}, (\delta_{k})_{k}, (\nu_{k})_{k}$ such that the following (technical) properties (whose relevance will become clear in later stages of our analysis) hold. First, we set 
\[
	\delta_{k} \coloneqq \frac{\log\log \tau_{k} + a_{k}}{\log \tau_{k}},
\]
for a sequence $(a_{k})_{k}$ taking values between $\log 6$ and $\log 6+1$ say, and a sequence $(\tau_k)_k$ to be defined below, and we define a sequence $(x_{k})_{k}$ via 
\begin{equation}
\label{eq: xk}
	\log\tau_{k} = \sqrt{\frac{\log x_{k}\log\log x_{k}}{2}}.
\end{equation}
The definition of the sequence $(x_{k})_{k}$ is of course reminiscent of the error term in \eqref{eq: N with exp error term}, and it is in fact on this sequence that the average $\int_{1}^{x}N_{C}(u)\dif u$ will display a desired deviation from the main term $\rho x^{2}/2$  as explained below (see \eqref{eq: omega for int N}). We also mention that $\tau_0$ will be assumed to be sufficiently large as needed in some of our future arguments.
Then, we require the following properties:
\begin{enumerate}[label = (\alph*)]
	\item $\tau_{k+1} > (2\tau_{k})^{5}$; \label{Property (a)}
	\item $(1+\delta_{k})\tau_{k}\log\tau_{k}\in 2\pi\Z$ and $\nu_{k}\tau_{k}\log\tau_{k} \in 2\pi\Z$; \label{Property (b)}
	\item $\tau_{k}\log x_{k} \in 2\pi\Z$ if $k$ is even while $\tau_{k}\log x_{k}\in \pi+2\pi\Z$ when $k$ is odd; \label{Property (c)}
	\item \[
		d\biggl(\frac{\log x_{k}}{(1+\delta_{k})\log\tau_{k}}\biggl(1-\frac{1+\sqrt{2}\sqrt{\log\log x_{k}/\log x_{k}}}{(1+\delta_{k})\log \tau_{k}}\biggr), \Z\biggr) < \frac{1/32}{(\log\tau_{k})^{3/4}}.
		\] \label{Property (d)}
\end{enumerate}
Here $d(\: \cdot\: , \Z)$ denotes the distance to the nearest integer. The existence of such sequences is stated in the ensuing lemma, whose proof will be postponed to Section \ref{section: conclusion proof lemma}.

\begin{lemma}
\label{lem: technical properties sequences}
There exist sequences $(\tau_{k})_{k}, (a_{k})_{k}, (\nu_{k})_{k}$ such that, with the above definitions of $(\delta_{k})_{k}$ and $(x_{k})_{k}$, the properties \ref{Property (a)}-\ref{Property (d)} are satisfied.
\end{lemma}
It is obvious that $\Pi_{C}$ satisfies \eqref{eq: asymp behavior number system 1}. The function $\Pi_{C}$ is indeed an absolutely continuous prime distribution function:

\begin{lemma}
The function $\Pi_{C}$ is absolutely continuous and non-decreasing.
\end{lemma}
\begin{proof}
That $\Pi_{C}$ is absolutely continuous is a simple consequence of its definition and Property \ref{Property (b)}. For $\tau_{k}^{1+\delta_{k}}< x <\tau_{k}^{\nu_{k}}$, 
\[ 
	\Pi_{C}'(x) = \frac{1-1/x}{\log x} + \frac{\tau_{k}\cos(\tau_{k}\log x)}{x} \ge \frac{1}{2\nu_{k}\log \tau_{k}} - \tau_{k}^{-\delta_{k}} \ge 0,
\] 
by the definition of $\delta_{k}$ and since $a_{k} \ge \log(2\nu_{k})$. Hence $\Pi_{C}$ is non-decreasing.
\end{proof}

 As in the statement of Theorem \ref{th: optimality HL}, we define the associated continuous integer distribution function $N_{C}$ via $\dif N_{C}=\exp^{\ast}(\dif \Pi_{C})$ and set
\[
	\zeta_{C}(s) \coloneqq \int_{1^{-}}^{\infty}x^{-s}\dif N_{C}(x)= \exp\left(\int_{1}^{\infty}x^{-s}\dif \Pi_{C}(x)\right).
\]
Then, we have using \ref{Property (b)},

\[
	\log \zeta_{C}(s) = \int_{1}^{\infty}x^{-s}\dif \Pi_{C}(x) = \log\biggl(\frac{s}{s-1}\biggr) 
	+ \sum_{k=0}^{\infty}\frac{1}{2}\biggl(\frac{\tau_{k}^{1-(1+\delta_{k})s}-\tau_{k}^{1-\nu_{k}s}}{s-\I\tau_{k}} +\frac{\tau_{k}^{1-(1+\delta_{k})s}-\tau_{k}^{1-\nu_{k}s}}{s+\I\tau_{k}}\biggr).
\]
This Mellin-Stieltjes transform is absolutely convergent for $\sigma>1$, and from the above formula, we immediately see that $\log \zeta_{C}$ has an analytic continuation to any simply connected region contained in $\sigma>0$ which does not contain $1$, and furthermore
$\zeta_{C}$ has a meromorphic continuation to $\sigma>0$ with a single simple pole at $s=1$. We remark that $\log \zeta_{C}$ indeed reaches the convexity bound (cf. \cite[Theorem 1.19, p.~201]{Tenenbaumbook}): if $t = \tau_{k}$, then $\log \zeta_{C} (\sigma+\I t) \gg t^{1-(1+\delta_{k})\sigma}$.

To go from the zeta function back to $N_{C}$, one uses Perron's inversion formula. For the sake of technical simplicity, we will employ\footnote{It might still be of interest to try to carry out the computation with the Perron formula for $N_{C}$ instead of the one for its primitive, since it appears that on the sequence $\tilde{x}_{k}$ defined via $\log\tau_{k} = \sqrt{\log\tilde{x}_{k} \log\log\tilde{x}_{k}}$, the contribution of the saddle points (see Section \ref{sec: The contribution from the saddle points}) is at least $\tilde{x}_{k}\exp(-2\sqrt{\log\tilde{x}_{k} \log\log\tilde{x}_{k}} + \text{ lower order })$. If the remainder of the Perron integral could be adequately estimated, this would improve the range for the constant $c$ in Theorem \ref{th: optimality HL} to any $c > 2$. } Perron's formula for the primitive of $N_{C}$, so that the integral converges absolutely.  For $\kappa>1$, we have
\[
	\int_{1}^{x}N_{C}(u)\dif u = \frac{1}{2\pi\I}\int_{\kappa-\I\infty}^{\kappa+\I\infty}\frac{x^{s+1}}{s(s+1)}\zeta_{C}(s)\dif s = \frac{1}{2\pi\I}\int_{\kappa-\I\infty}^{\kappa+\I\infty}\frac{x^{s+1}}{(s-1)(s+1)}\exp\biggl(\sum_{k}\dotso\biggr)\dif s.
\]

We shift the contour of integration to the left, more specifically to the contour considered by Hilberdink and Lapidus in their proof of Theorem \ref{th: HilberdinkLapidus}. Set 
$$\GHL \coloneqq \{\, 1-1/\e + \I t: 0<t<\e^{\e}\,\} \cup \{ \, 1-\log\log t/\log t + \I t: t>\e^{\e}\,\}.$$
Set $\rho \coloneqq \res_{s=1}\zeta_{C}(s)$. By the residue theorem, we get a contribution from the pole at $s=1$:
\begin{equation}
\label{eq: Perron formula}
	\int_{1}^{x}N_{C}(u)\dif u = \frac{\rho}{2}x^{2} + \frac{1}{\pi}\Im \int_{\GHL} \frac{x^{s+1}}{(s-1)(s+1)}\exp\biggl(\sum_{k}\dotso\biggr)\dif s,
\end{equation}
where we have used $\zeta_{C}(\overline{s})=\overline{\zeta_{C}(s)}$ to restrict the path of integration to the upper half plane. In order to estimate the remaining integral, we will exploit the fact that $\tau_{k}^{1-(1+\delta_{k})s}/(s-\I\tau_{k})$ becomes relatively small when $t$ is far from $\tau_k$. Specifically, for each $k$ we will choose a suitable $x$ (namely $x=x_{k}$ defined above in \eqref{eq: xk}) so that the integral near $t=\tau_{k}$ will give a contribution of order $x\exp(-c\sqrt{\log x\log\log x})$, and so that the rest of the integral is of lower order. Making this explicit is however a technically challenging problem. One of the difficulties that arises comes from taking the exponential: lower bounds on $\abs{\log\zeta_{C}}$ do not necessarily imply lower bounds on $\abs{\zeta_{C}}$, and furthermore, exponentiation also introduces a lot of oscillation. In order to extract the contribution of the integral near $t=\tau_{k}$, we will use the saddle-point method.

Let us briefly review some ideas connected with the saddle-point method. Given a region $\Omega\subseteq \C$, a contour $\Gamma\subseteq\Omega$, and analytic functions $f$ and $g$  on $\Omega$, one might proceed as follows to estimate the integral $\int_{\Gamma}g(z)\e^{f(z)}\dif z$. First, one computes the saddle points of $f$; these are the points $s_{j}\in\Omega$ for which $f'(s_{j})=0$. Near a saddle point, the graph of $\Re f$ looks like a saddle surface. The idea is to shift the contour $\Gamma$ while fixing the endpoints to a contour $\tilde{\Gamma}$ which passes through the saddle points in such a way that on the new contour $\Re f(s)$ reaches a maximum at the saddle points -- whether this is possible of course depends  on the specific situation.  
Approximating $f$ by its second order Taylor polynomial near saddle points, one gets the following approximation:
\[	
	\int_{\Gamma}g(z)\e^{f(z)}\dif z \approx \sum_{j} g(s_{j})\e^{f(s_{j})}\int\e^{f''(s_{j})(s-s_{j})^{2}/2}\dif s \approx \sum_{j} \sqrt{\frac{2\pi}{-f''(s_{j})}}g(s_{j})\e^{f(s_{j})}.
\]
Often, one wants to estimate $\int_{\Gamma}g(z)\e^{\lambda f(z)}\dif z$ for a parameter $\lambda$ tending to $\infty$, and under certain assumptions, one can deduce that 
\[
	\int_{\Gamma}g(z)\e^{\lambda f(z)}\dif z \sim \sum_{j}\sqrt{\frac{2\pi}{-\lambda f''(s_{j})}}g(s_{j})\e^{\lambda f(s_{j})}, \quad \text{ as } \lambda \to \infty.
\]
We refer to \cite[Chapters 5 and 6]{deBruijn} for a classical account of the method, and to \cite[Section 3.6]{EstradaKanwalbook} for a distributional approach to this technique.

In our case, we will apply the saddle-point method to a portion of the integral near $t=\tau_{k}$ with 
\[
	f(s) = f_{k}(s) \coloneqq (s+1)\log x_{k} + \frac{1}{2}\frac{\tau_{k}^{1-(1+\delta_{k})s}}{s-\I\tau_{k}}, \quad g(s) = g_{k}(s) \coloneqq \frac{\exp\left(\sideset{}{'}\sum \cdots\right)}{(s-1)(s+1)}, 
\]
where we use the notation $\sideset{}{'}\sum$ to indicate that we exclude the term $\tau_{k}^{1-(1+\delta_{k})s}/(2(s-\I\tau_{k}))$ from the summation. In the absence of a parameter $\lambda$ as in the standard situation mentioned above, we are led to make a rather explicit and detailed analysis of the integral term appearing in the right-hand side of \eqref{eq: Perron formula}.

Summarizing, we will get
\[
	\int_{1}^{x}N_{C}(u)\dif u = \frac{\rho}{2}x^{2} + \text{ contribution from saddle points } + \text{ remainder }.
\]
In Section \ref{sec: The contribution from the saddle points}, we will deal with the saddle points and show that their contribution is $\gg x^{2}\exp(-c\sqrt{\log x\log\log x})$. In Section \ref{sec: The remainder}, we will show that the remaining part of the Perron integral is of strictly lower order. 
 Summing up all results, we will prove that
 \begin{equation}
\label{eq: omega for int N}
	\int_{1}^{x}N_{C}(u)\dif u = \frac{\rho}{2}x^{2} + \Omega_{\pm}\bigl(x^{2}\exp\bigl(-c\sqrt{\log x\log\log x}\bigr)\bigr)
\end{equation}
for any $c>2\sqrt{2}$. The relation \eqref{eq: omega for int N} readily implies \eqref{eq: asymp behavior number system 2}, so this will finally establish Theorem \ref{th: optimality HL}.

From now on, we fix a specific $k$, and investigate the Perron integral for $x=x_{k}$ given by \eqref{eq: xk}. For ease of notation, we will drop the index $k$ everywhere, unless we need to make the distinction between the specific $\tau_{k}$ and the other $\tau_{j}$, $j\neq k$.

\section{The contribution from the saddle points}
\label{sec: The contribution from the saddle points}
\subsection{The saddle points}
Recall that we have set 
\[
	f(s) =  (s+1)\log x + \frac{1}{2}\frac{\tau^{1-(1+\delta)s}}{s-\I\tau}, \quad \log \tau = \sqrt{\frac{\log x\log\log x}{2}}.
\]
Also,
\[\delta=\frac{\log\log\tau +a}{\log \tau}.
\]
We have
\begin{equation}
\label{eq: f'}
	f'(s) = \log x - \frac{1}{2}\frac{\tau^{1-(1+\delta)s}}{s-\I\tau}\biggl((1+\delta)\log \tau + \frac{1}{s-\I\tau}\biggr).
\end{equation}

We will show that $f$ has a saddle point on the line $t=\tau$. Due to the periodicity of $\tau^{1-(1+\delta)s}$, there will also be saddle points near $t = \tau + 2\pi m /((1+\delta)\log\tau)$, when $m\in\Z$ is not too large. Making this explicit, set 
\begin{equation}
\label{eq: t+-m}
	t_{m}^{+}\coloneqq \tau + \frac{2\pi m + \pi/2}{(1+\delta)\log\tau}, \quad t_{m}^{-}\coloneqq \tau + \frac{2\pi m -\pi/2}{(1+\delta)\log\tau},	
\end{equation}
and set $V_{m}$ to be the rectangle with vertices $1/2 + \I t_{m}^{\pm}$, $1 + \I t_{m}^{\pm}$, $m\in \Z$.

\begin{lemma}\label{lemma: uniqueness saddle point}
Suppose $\abs{m} < \eps \log\tau$ for some sufficiently small (fixed) $\eps>0$. Then $f$ has a unique saddle point (of multiplicity 1) in the interior of $V_{m}$. 
\end{lemma}
\begin{proof}
Starting with the vertical edge on the right, and continuing in a counter clockwise fashion, we name the edges of $\partial V_{m}$ as $E_{j}$, $j=1, \ldots, 4$. By the assumption on $m$, $\arg(s-\I\tau) <\eps'$ for some small $\eps'$ when $s\in \partial V_{m}$. On the first segment $E_{1}$, the first term in \eqref{eq: f'} is dominant and so the argument of $f'$ is close to 0. On $E_{2}$, the second term becomes more and more significant, and the argument increases from about 0 to about $\pi/2$. On $E_{3}$ the argument increases further from about $\pi/2$ to about $3\pi/2$. Finally, on $E_{4}$, the argument increases further from about $3\pi/2$ to about $2\pi$, as the first term becomes once again dominant. Hence the winding number of the curve $f'(\partial V_{m})$ around the origin is 1, and the lemma follows from the argument principle.
\end{proof}

Suppose\footnote{The choice for the exponent $3/4$ is a bit arbitrary. It allows us to achieve error terms of decent quality with relatively modest effort (compared to the bound $\varepsilon \log \tau$). Later on, we will impose a stronger restriction on $m$, which will appear naturally. We have chosen not to impose this restriction here already and to begin with this rather arbitrary one, because we think the intrinsic nature of the additional restriction might get clouded otherwise.} now that $\abs{m} < (\log\tau)^{3/4}$. By the above lemma, for every such $m$, $f$ has a unique saddle point $s_{m}$ in the rectangle $V_{m}$, which is of multiplicity 1. We have $f'(s_{m})=0$, which is equivalent to
\begin{equation}
\label{eq: saddle point identity}
\frac{1}{2}\frac{\tau^{1-(1+\delta)s_{m}}}{s_{m}-\I\tau} = \frac{\log x}{(1+\delta)\log\tau + \frac{1}{s_{m}-\I\tau}}.
\end{equation}
By taking logarithms, one sees that for each $m$, there is an integer $n_{m}$ such that the following implicit equations for the real and imaginary part of $s_{m}= \sigma_{m} + \I t_{m}$ hold:
\begin{subequations}
	\begin{equation}
	\label{eq: saddle point eq sigma}
	\sigma_{m} = \frac{1}{1+\delta}\biggl(1 - \frac{1}{\log\tau}\biggl(\log\log x + \log 2 + \log\abs{s_{m}-\I\tau}-\log\abs{(1+\delta)\log\tau + \frac{1}{s_{m}-\I\tau}}\biggr)\biggr), 
	\end{equation}
	\begin{equation}
	\label{eq: saddle point eq t}
	t_{m} = \frac{1}{(1+\delta)\log\tau}\biggl(\arg\biggl((1+\delta)\log\tau+\frac{1}{s_{m}-\I\tau}\biggr) - \arg(s_{m}-\I\tau) + 2\pi n_{m}\biggr). 
	\end{equation}
\end{subequations}

Let us first look at the equation \eqref{eq: saddle point eq sigma} for $\sigma_{m}$. Since $s_{m}\in V_{m}$ and $\abs{m}<(\log\tau)^{3/4}$, $\log\abs{s_{m}-\I\tau} \ll 1$ and $\log\abs{(1+\delta)\log\tau + (s_{m}-\I\tau)^{-1}} = \log\log\tau + O(1)$. Inserting this in \eqref{eq: saddle point eq sigma} and using a Taylor approximation for $1/(1+\delta)$, we see that
\begin{align*}
	\sigma_{m} 	&= \biggl(1-\frac{\log\log\tau + a}{\log\tau} + O\biggl(\biggl(\frac{\log\log\tau}{\log\tau}\biggr)^{2}\biggr)\biggr)\biggl(1-\frac{1}{\log\tau}\biggl(\log\log x - \log\log\tau + O(1)\biggr)\biggr)\\
				&= 1 - \sqrt{2}\sqrt{\frac{\log\log x}{\log x}} + O\biggl(\frac{1}{\sqrt{\log x\log\log x}}\biggr).
\end{align*}
Using this approximation for $\sigma_{m}$ (and again $\abs{m}<(\log\tau)^{3/4}$), one sees that $\log\abs{s_{m}-\I\tau} \ll (\log x\log\log x)^{-1/8}$ and $\log\abs{(1+\delta)\log\tau + (s_{m}-\I\tau)^{-1}} = \log\log\tau + O((\log\log x/\log x)^{1/2})$, so that 
\begin{equation}
\label{eq: approx sigma}
	\sigma_{m} = 1 - \sqrt{2}\sqrt{\frac{\log\log x}{\log x}} - \frac{\sqrt{2}(a+\log 2)}{\sqrt{\log x\log\log x}} + O \biggl( \frac{1}{(\log x\log\log x)^{5/8}}\biggr).
\end{equation}
Also, repeating the argument for $\sigma_{0}$ we obtain
\begin{equation}
\label{eq: approx sigma_{0}}
	\sigma_{0} = 1 - \sqrt{2}\sqrt{\frac{\log\log x}{\log x}} - \frac{\sqrt{2}(a+\log 2)}{\sqrt{\log x\log\log x}} + O\biggl(\frac{\log\log x}{\log x}\biggr).
\end{equation}

Let us now look at equation \eqref{eq: saddle point eq t} for $t_{m}$. By assumption \ref{Property (b)}, $(1+\delta)\tau\log\tau = 2\pi M$ for some integer $M$. Hence, we see that $t_{0}=\tau$ satisfies the equation with $n_{0}=M$. To see that indeed $\Im s_{0} = \tau$, one can check using a continuity argument that equation \eqref{eq: saddle point eq sigma} for $m=0$ and $\tau=t_{0}$ has a solution $\sigma_{0}$ between $1/2$ and $1$. The point $\sigma_{0}+\I\tau$ then satisfies both \eqref{eq: saddle point eq sigma} and \eqref{eq: saddle point eq t}, and by uniqueness (Lemma \ref{lemma: uniqueness saddle point}), we must have $s_{0}=\sigma_{0}+\I\tau$. For general $m$, different from 0, we will again approximate the solutions. The arguments appearing in \eqref{eq: saddle point eq t} can be written as follows:
\begin{align*}
\alpha_{m}&\coloneqq \arg(\sigma_{m} + \I(t_{m}-\tau)) = \arctan\biggl(\frac{t_{m}-\tau}{\sigma_{m}}\biggr), \\
\beta	_{m}	&\coloneqq \arg\biggl((1+\delta)\log\tau + \frac{1}{\sigma_{m} + \I(t_{m}-\tau)}\biggr) = \arctan\biggl(\frac{-\frac{t_{m}-\tau}{\sigma_{m}^{2}+(t_{m}-\tau)^{2}}}{(1+\delta)\log\tau + \frac{\sigma_{m}}{\sigma_{m}^{2}+(t_{m}-\tau)^{2}}}\biggr).
\end{align*}
Using the bound $\arctan x \ll \abs{x}$, and the fact that $t_{m}\in V_{m}$  ($\:\abs{m}<(\log\tau)^{3/4}$), we find that $n_{m} = M + m$, and we get a first approximation for $t_{m}$: 
\[
	t_{m} = \tau + \frac{2\pi m}{(1+\delta)\log\tau} + O\biggl(\frac{1}{(\log\tau)^{5/4}}\biggr).
\]
Employing the above approximation and \eqref{eq: approx sigma}, we can get finer estimates for $\alpha_{m}$ and $\beta_{m}$:
\begin{align*}
	\alpha_{m} 	&= \frac{1}{\sigma_{m}}\frac{2\pi m}{(1+\delta)\log\tau} + O\biggl(\frac{1}{(\log\tau)^{5/4}}\biggr) + O\biggl(\frac{\abs{m}^{3}}{(\log\tau)^{3}}\biggr)\\
			&= \frac{2\pi m}{(1+\delta)\log\tau} + O\biggl(\frac{\abs{m}}{\log\tau}\sqrt{\frac{\log \log x}{\log x}} + \frac{1}{(\log\tau)^{5/4}} + \frac{\abs{m}^{3}}{(\log\tau)^{3}} \biggr),\\
	\beta_{m}		&= O\biggl(\frac{1}{(\log\tau)^{5/4}}\biggr).		
\end{align*} 
They, in turn, yield a better asymptotic estimate for $t_{m}$:
\[
	t_{m} = \tau + \frac{2\pi m}{(1+\delta)\log\tau}\biggl(1-\frac{1}{(1+\delta)\log\tau}\biggr) + O\biggl( \frac{1}{(\log\tau)^{9/4}}+ \frac{\abs{m}\log\log\tau}{(\log\tau)^{3}} + \frac{\abs{m}^{3}}{(\log\tau)^{4}}\biggr).
\]
Repeating the procedure one final time, we obtain
\begin{align*}
	\alpha_{m}	&= \frac{1}{\sigma_{m}}\frac{2\pi m}{(1+\delta)\log\tau}\biggl(1-\frac{1}{(1+\delta)\log\tau}\biggr) + O\biggl(\frac{1}{(\log\tau)^{9/4}} + \frac{\abs{m}\log\log\tau}{(\log\tau)^{3}} + \frac{\abs{m}^{3}}{(\log\tau)^{3}}\biggr)\\
			&= \frac{2\pi m}{(1+\delta)\log\tau}\biggl(1+\sqrt{2}\sqrt{\frac{\log\log x}{\log x}}\biggr) + O\biggl(\frac{\abs{m}}{(\log\tau)^{2}} + \frac{\abs{m}^{3}}{(\log\tau)^{3}}\biggr),\\
	\beta_{m}		&= -\frac{2\pi m}{(1+\delta)^{2}(\log\tau)^{2}} + O\biggl(\frac{\abs{m}\log\log\tau}{(\log\tau)^{3}}+ \frac{\abs{m}^{3}}{(\log\tau)^{4}}\biggr),		
\end{align*}
so that 
\begin{equation}
\label{eq: approx t}
t_{m} = \tau + \frac{2\pi m}{(1+\delta)\log\tau}\biggl(1 - \frac{1+\sqrt{2}\sqrt{\frac{\log\log x}{\log x}}}{(1+\delta)\log\tau}\biggr) + O\biggl(\frac{\abs{m}}{(\log\tau)^{3}} + \frac{\abs{m}^{3}}{(\log\tau)^{4}}\biggr).
\end{equation}

For future computations, it is useful to have an approximation for $f$ and $f'$ near the saddle points.
\begin{lemma}
\label{lem: approx f and f'}
There are continuous functions $\lambda_{m}(s)$ and $\tilde{\lambda}_{m}(s)$ such that
\begin{align*}
	f(s) 	&= f(s_{m}) + \frac{f''(s_{m})}{2}(s-s_{m})^{2}(1+\lambda_{m}(s)),\\
	f'(s) 	&= f''(s_{m})(s-s_{m})(1+\tilde{\lambda}_{m}(s)), 
\end{align*}
and with the property that for each $\eps>0$ there exists an $\eta>0$ independent of $\tau$ and $m$ such that 
\[
	\abs{s-s_{m}} < \frac{\eta}{\log\tau} \implies \abs{\lambda_{m}(s)} + \abs[1]{\tilde{\lambda}_{m}(s)} < \eps.
\] 
\end{lemma}
\begin{proof}
We will show the assertion for $f$, the proof of the statement concerning $f'$ is similar. As $f$ is analytic in a neighborhood of $s_{m}$ and $f'(s_{m}) = 0$, we have, for $s$ sufficiently close to $s_{m}$,
\[
 \left|f(s) - f(s_{m}) - \frac{f''(s_{m})}{2}(s-s_{m})^{2}\right| \leq \frac{M_{f,s,m} |s-s_{m}|^{3}}{6}, 
\]
where $M_{f,s,m}$ is the maximum of $|f^{(3)}|$ on the line between $s_{m}$ and $s$.
The derivatives of $f$ for $n\geq 2$ are
\[
	f^{(n)}(s) = \frac{1}{2}\tau^{1-(1+\delta)s}\sum_{l=0}^{n}\binom{n}{l}\frac{(-1)^{l}l!}{(s-\I\tau)^{l+1}}(-1)^{n-l}((1+\delta)\log\tau)^{n-l}.
\]
Therefore, since $s_{m}-\I\tau \sim 1$,
\[
 |\lambda_{m}(s)| \leq \frac{M_{f,s,m} |s-s_{m}|}{3|f''(s_{m})|} \ll |s-s_{m}| \log \tau, 
\]
where the implicit constant is independent of $m$ and $\tau$. The continuity of $\lambda_{m}$ is obvious.
\end{proof}

\subsection{The paths of steepest descent}
In order to estimate the contribution of the saddle points, we shall shift, near each saddle point, the contour to the so-called path of steepest descent. This is a contour through the saddle point which, when starting at the saddle point, displays the biggest decrease in $\Re f(s)$ among all possible paths. Intuitively, this path connects the two ``valleys'' on both sides of the saddle point in the most economical way. Starting in one of the ``valleys'', the tangent vector along this path is at first a positive multiple of $\nabla \Re f(s)$, as $\Re f(s)$ increases to a maximum at the saddle point. After passing the saddle point, the tangent vector along this path is a positive multiple of $-\nabla \Re f(s)$, as $\Re f(s)$ decreases. Using the Cauchy-Riemann equations, one sees that $\Im f(s)$ is constant along this path. It is worth mentioning that there is another path through the saddle point on which $\Im f(s)$ is constant, namely the path of steepest \emph{ascent} (which displays the opposite behavior of the \emph{descent} path).

We will show that for each $m$, there is a path of steepest descent which goes from the bottom horizontal edge of $V_{m}$ to its upper horizontal edge. The situations for $m=0$ and $m\neq0$ are a bit different; let us first describe it for $m=0$.

Write $\theta = (t-\tau)(1+\delta)\log\tau$, so that $\theta$ varies between $-\pi/2$ and $\pi/2$ as $t$ varies between $t_{0}^{-}$ and $t_{0}^{+}$; see \eqref{eq: t+-m}. The equation $\Im f(s) = \Im f(s_{0})$ is equivalent to
\begin{equation}
\label{eq: steepest path 0}
	\frac{1}{2}\frac{\tau^{1-(1+\delta)\sigma}}{\sigma^{2}+ (t-\tau)^{2}}\bigl((t-\tau)\cos\theta + \sigma\sin\theta\bigr) = (t-\tau)\log x.
\end{equation}
Trivial solutions are given by the line $s=\sigma+\I\tau$. It is however readily seen that this line is the path of steepest ascent, by examining the behavior of $\Re f$. Consider now $t\neq \tau$ fixed (hence also $\theta\neq 0$ fixed). Then, the equation \eqref{eq: steepest path 0} has a unique solution for $\sigma$ in the range $1/2<\sigma<1$: first of all $\sgn(\LHS)=\sgn(\RHS)$, and second we have that $\abs{\LHS}$ is monotonically decreasing in that range, with $\abs{\LHS(\sigma=1)}\ll |t-\tau|$, and $\abs{\LHS(\sigma=1/2)}\gg \tau^{(1-\delta)/2}|t-\tau|$ when $\theta\le \pi/4$ and $\abs{\LHS(\sigma=1/2)}\gg \tau^{(1-\delta)/2}$ when $\theta \ge \pi/4$. This gives the existence of another path of constant imaginary part through $s_{0}$, which crosses the horizontal edges of $V_{0}$, and which is necessarily the path of steepest descent.

The situation in the case $m\neq 0$ is less straightforward. Consider the case $m>0$ (the case $m<0$ is analogous). It is convenient to write 
\[
\theta = (t-\tau)(1+\delta)\log\tau - 2\pi m
\]
so that, as $t$ varies between $t_{m}^{-}$ and $t_{m}^{+}$, the quantity $\theta$ varies between $-\pi/2$ and $\pi/2$; see again \eqref{eq: t+-m}. Using \eqref{eq: saddle point identity}, the equation $\Im f(s) = \Im f(s_{m})$ is equivalent to
\begin{equation}
\label{eq: steepest path m}
	\frac{1}{2}\frac{\tau^{1-(1+\delta)\sigma}}{\sigma^{2}+(t-\tau)^{2}}\bigl((t-\tau)\cos\theta + \sigma\sin\theta\bigr) = (t-t_{m})\log x - v_{m},
\end{equation}
where
\[
	v_{m}\coloneqq \Im \frac{\log x}{(1+\delta)\log\tau + \frac{1}{s_{m}-\I\tau}} = \log x\frac{\frac{t_{m}-\tau}{\sigma_{m}^{2}+(t_{m}-\tau)^{2}}}{\bigl((1+\delta)\log\tau + \frac{\sigma_{m}}{\sigma_{m}^{2}+(t_{m}-\tau)^{2}}\bigr)^{2} + \bigl(\frac{t_{m}-\tau}{\sigma_{m}^{2}+(t_{m}-\tau)^{2}}\bigr)^{2}}.
\]
Note that
\begin{equation}
\label{eq: bound v_{m}}
	v_{m} \ll m \frac{\log x}{(\log \tau)^{3}}.
\end{equation}
Denote the left hand side of \eqref{eq: steepest path m} by $l_{t}(\sigma)$, and the right hand side by $r_{t}$.

In Figure 1, the qualitative behavior of the steepest decent/ascent paths are depicted, as well as the relative positions of these paths with respect to some other points and curves. The points $t_{m}^{(1)}$ and $t_{m}^{(2)}$ are defined as the solution of $r_{t} = 0$, i.e., $t_{m}^{(1)}= t_{m} + v_{m}/\log x$, respectively $t_{m}^{(2)} = \tau+2\pi m/((1+\delta)\log\tau)$, corresponding to $\theta=0$. (Note that \eqref{eq: approx t} and \eqref{eq: bound v_{m}} imply $t_{m}^{-}<t_{m}<t_{m}^{(1)}<t_{m}^{(2)}<t_{m}^{+}$.) The blue line is the set of points for which $l_{t}(\sigma) = 0$.

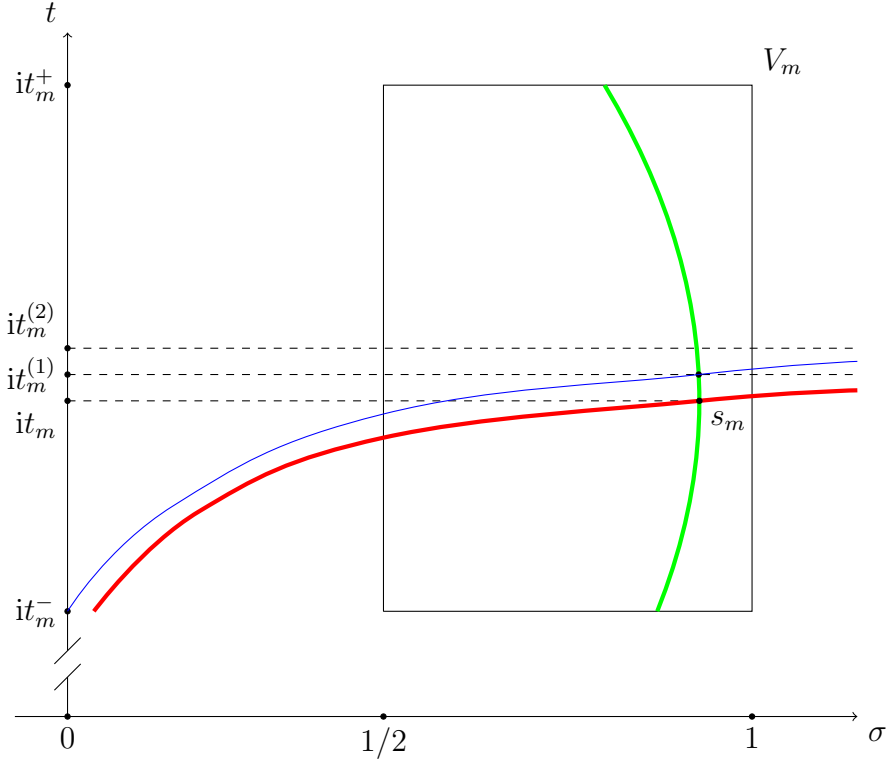
\begin{figure}
\begin{tikzpicture}[scale=0.7]
	\draw[->] (1,1) -- (1,12);
	\node[above left] at (1,12) {$t$};
	\draw[->] (0,-1) -- (16,-1);
	\node[below right] at (16,-1) {$\sigma$};	
	\draw[-] (1,-1) -- (1,-0.25);
	\draw[-] (0.75,-0.5) -- (1.25,0);
	\draw[-] (1,0.25) -- (1,1);
	\draw[-] (0.75, 0) -- (1.25,0.5);
	
	\draw[fill] (1,1) circle [radius=0.05];
	\node[left] at (1,1) {$\I t_{m}^{-}$};
	\draw[fill] (1,5) circle [radius=0.05];
	\node[below left] at (1,5) {$\I t_{m}$};
	\draw[dashed] (1,5) -- (12.975,5);
	\draw[fill] (1,5.5) circle [radius=0.05];
	\node[left] at (1,5.5) {$\I t_{m}^{(1)}$};
	\draw[dashed] (1,5.5) -- (16, 5.5);
	\draw[fill] (1,6) circle [radius=0.05];
	\node[above left] at (1,6) {$\I t_{m}^{(2)}$};
	\draw[dashed] (1,6) -- (16,6);
	\draw[fill] (1,11) circle [radius=0.05];
	\node[left] at (1,11) {$\I t_{m}^{+}$};
	\draw[fill] (1,-1) circle [radius=0.05];
	\node[below] at (1,-1) {$0$};
	\draw[fill] (14,-1) circle [radius=0.05];
	\node[below] at (14,-1) {$1$};
	\draw[fill] (7,-1) circle [radius=0.05];
	\node[below] at (7,-1) {$1/2$};
	
	 \draw (7,1) -- (14,1) -- (14,11) -- (7,11) -- (7,1);
	 \node[above right] at (14,11) {$V_{m}$};
	\draw[ultra thick, variable=\y, domain=1:11, green] plot({-(0.05)*(\y - 5)*(\y - 5)+13}, {\y});
	\draw[ultra thick, red] plot [smooth, tension=0.7] coordinates {(1.5,1) (3.5,2.9) (7, 4.3) (13,5) (16,5.2)};
	\draw[fill] (13, 5) circle [radius=0.05];
	\draw[fill] (12.9875,5.5) circle [radius=0.05];
	\node[below right] at (13, 5) {$s_{m}$};
	\draw[blue] plot [smooth, tension=0.75] coordinates {(1,1) (3,3) (7, 4.75)  (12.9875,5.5) (16,5.75)}; 
\end{tikzpicture}
\caption{Qualitative behavior of steepest paths for $m>0$ (green for descent, red for ascent). The blue line is given by $\sigma=(t-\tau)\cot \abs{\theta}$.}
\end{figure}

One can show the existence of the path of steepest descent by a careful analysis of \eqref{eq: steepest path m}. It is however more convenient to use the approximations for $f$ and $f'$ given in Lemma \ref{lem: approx f and f'} to show this existence, at least in a sufficiently small neighborhood of the saddle point. 
We have
\begin{align}
	f''(s_{m}) 	&= \log x\biggl((1+\delta)\log\tau + \frac{1}{s_{m}-\I\tau} + \frac{1}{(s_{m}-\I\tau)^{2}\bigl((1+\delta)\log\tau + \frac{1}{s_{m}-\I\tau}\bigr)}\biggr) \nonumber \\
			&= \log x\bigl((1+\delta)\log\tau+O(1)\bigr), \label{eq: f''}
\end{align}
so $\arg f''(s_{m}) \ll 1/\log\tau$. Now fix an $\eps>0$ sufficiently small\footnote{How small we need $\eps$ to be will be determined later, but it is important to note that $\eps$ and later also $\eta$ can be chosen independently of $\tau$ and $m$.} and write $s-s_{m} = r\e^{\I\phi}$. By Lemma \ref{lem: approx f and f'}, there are a function $\lambda_{m}$ and $\eta>0$ such that for $r<\eta/\log\tau$, we have
\[
	f(s) = f(s_{m}) + \frac{f''(s_{m})}{2}r^{2}\e^{2\I\phi}(1+\lambda_{m}(s)), \quad \abs{\lambda_{m}(s)} < \eps.
\]
Set $g_{m}(s)\coloneqq f''(s_{m})r^{2}\e^{2\I\phi}(1+\lambda_{m}(s))/2$. The path of steepest descent is given by $\Im g_{m}=0$ under the constraint $\Re g_{m}< 0$. Since
\[
	\Re g_{m}(s) = \frac{\abs{f''(s_{m})}}{2}r^{2}\bigl((1+\Re\lambda_{m}(s))\cos2\phi - (\Im \lambda_{m}(s))\sin2\phi + O(1/\log\tau)\bigr),
\]
we see that by choosing $\eps$ sufficiently small  and $\tau$ large enough we must necessarily have $\phi \in (-7\pi/8, -\pi/8) \cup (\pi/8, 7\pi/8)$ in order to satisfy the condition $\Re g_{m}(s) < 0$. On the other hand,
\[
	\Im g_{m}(s) = \frac{\abs{f''(s_{m})}}{2}r^{2}\bigl((1+\Re \lambda_{m}(s))\sin2\phi + (\Im\lambda_{m}(s))\cos2\phi + O(1/\log\tau)\bigr).
\]
For each $r<\eta/\log\tau$, there is at least one solution $\phi$ near $\pi/2$ and at least one solution near $-\pi/2$ of the equation $\Im g_{m}(r\e^{\I\phi})=0$. For example, by selecting $\eps$ sufficiently small  and $\tau$ large enough, there is a solution for $\phi$ in $(2\pi/5, 3\pi/5)$ and in $(-3\pi/5, -2\pi/5)$.

This guarantees the existence of the path of steepest descent in the range $\abs{\theta}\le \eta/2$ say (since $t_{m} = t_{m}^{(2)}+O(1/(\log\tau)^{5/4})$). Note that this part of the path lies inside the rectangle $V_{m}$, since $B(s_{m}, \eta/\log\tau) \subseteq V_{m}$ for sufficiently large $\tau$. When $\theta>\eta/2$, $r_{t}$ is positive; hence, by monotonicity of $l_{t}(\sigma)$ and calculating the values at $\sigma=1/2$ and $\sigma=1$, we deduce that for every $\theta\in (\eta/2, \pi/2]$ there is a unique solution $\sigma \in (1/2, 1)$ of  \eqref{eq: steepest path m}. We conclude that we can extend the path of steepest descent from $\theta=\eta/2$ upwards to $\theta=\pi/2$, corresponding to $t=t_{m}^{+}$.

It remains to treat the case $\theta <-\eta/2$. Then $r_{t}<0$, so that \eqref{eq: steepest path m} can only have solutions when $\sigma > (t-\tau)\cot\abs{\theta}$, and we remark that for $\theta < -\eta/2$ and sufficiently large $\tau$, $(t-\tau)\cot\abs{\theta} < 1/2$. Since $l_{t}((t-\tau)\cot \abs{\theta})=0$ and $l_{t}(\sigma) \to 0$ as $\sigma\to\infty$, $l_{t}$ has at least one (local) minimum for $\sigma> (t-\tau)\cot\abs{\theta}$. In fact, $l_{t}$ has precisely one minimum: 
\begin{equation*}
l_{t}'(\sigma) = \frac{1}{2}\frac{\tau^{1-(1+\delta)\sigma}}{\sigma^{2}+(t-\tau)^{2}}(\sin\theta - \psi_{t}(\sigma)),
\end{equation*}
with
\[
	\psi_{t}(\sigma) \coloneqq \bigl((t-\tau)\cos\theta + \sigma\sin\theta\bigr)\biggl((1+\delta)\log\tau+\frac{2\sigma}{\sigma^{2}+(t-\tau)^{2}}\biggr).
\]
Now 
\[
	\psi'_{t}(\sigma) = \sin\theta\biggl((1+\delta)\log\tau + \frac{2\sigma}{\sigma^{2}+(t-\tau)^{2}}\biggr) + \bigl((t-\tau)\cos\theta + \sigma\sin\theta\bigr)\frac{2(\sigma^{2}+(t-\tau)^{2})-4\sigma^{2}}{(\sigma^{2}+(t-\tau)^{2})^{2}},
\]
and the second term is bounded in absolute value by 
\[
	\frac{\abs{2\sigma^{2}-2(t-\tau)^{2}}}{(\sigma^{2}+ (t-\tau)^{2})^{2}}2\sigma\abs{\sin\theta}, \quad \text{if } \sigma \ge (t-\tau)\cot\abs{\theta}.
\]
This implies that $\psi_{t}$ is monotonically decreasing for $\sigma > (t-\tau)\cot\abs{\theta}$, so that $l_{t}$ has a unique local extremum in this range, which must be a minimum. Now $\abs{r_{t}} \asymp_{\eta} \log x/\log\tau$, while $\abs{l_{t}(1/2)} \gg_{\eta} \tau^{(1-\delta)/2}$, so that for each fixed $\theta < -\eta/2$, \eqref{eq: steepest path m} has two distinct solutions for $\sigma$, one on either side of $1/2$. In particular\footnote{The solution to the left of $1/2$ corresponds to the path of steepest ascent.}, we can extend the path of steepest descent from $\theta=-\eta/2$ (where it was on the right of $1/2$) downwards to $\theta=-\pi/2$, corresponding to $t=t_{m}^{-}$. This path cannot cross the line $\sigma=1$, since $\abs{l_{t}(1)} \ll 1/\log\tau$, which is of strictly lower order than $r_{t}$.

\bigskip

Now that we have shown the existence of the paths of steepest descent, we will also deduce some information about the argument of the tangent vector along these paths. Denote the steepest path by $\Gamma_{m}$, and let $\gamma_{m}\colon [y_{m}^{-}, y_{m}^{+}] \to \Gamma_{m}$ be a unit speed parametrization of $\Gamma_{m}$ (i.e. $\abs{\gamma'_{m}}=1$) with $\Im\gamma_{m}(y_{m}^{\pm}) = t_{m}^{\pm}$ and $\gamma_{m}(0) = s_{m}$. We have that $\gamma'_{m}(y)$ is a positive multiple of $\nabla \Re f(\gamma_{m}(y)) = \overline{f'}(\gamma_{m}(y))$ for $y<0$, while $\gamma'_{m}(y)$ is a negative multiple of $\nabla \Re f(\gamma_{m}(y)) = \overline{f'}(\gamma_{m}(y))$ for $y>0$.

Write again $s-s_{m} = r\e^{\I\phi}$. From the above discussion, we know that for $\abs{\theta}<\eta/2$ this path lies in the cone $\phi \in (-3\pi/5, -2\pi/5) \cup (2\pi/5, 3\pi/5)$. By appealing to Lemma \ref{lem: approx f and f'} once more, we see that there exists a function $\tilde{\lambda}_{m}$ such that  
\[
	f'(s) = f''(s_{m})r\e^{\I\phi}(1+ \tilde{\lambda}_{m}(s))
\]
with $\abs[0]{\tilde{\lambda}_{m}(s)} < \eps$ for $r<\eta/\log \tau$. From this we deduce that $\abs[1]{\arg(\e^{-\I\pi/2} \gamma'_{m}(y))} < \pi/5$ in the range $\abs{\theta}<\eta/2$, provided we choose $\eps$ sufficiently small and $\tau$  sufficiently large.

For the range $\eta/2<\abs{\theta}\leq \pi /2$, we estimate the argument of $\overline{f'}$ directly. We have
\[
	\overline{f'}(s) = \log x + \frac{1}{2}\frac{\tau^{1-(1+\delta)\sigma}}{\sigma^{2}+(t-\tau)^{2}}(A-\I B),
\]
with
\begin{align*}
A 	&\coloneqq  \bigl((t-\tau)\sin\theta - \sigma\cos\theta\bigr)\biggl((1+\delta)\log\tau + \frac{\sigma}{\sigma^{2}+(t-\tau)^{2}}\biggr) + \frac{\bigl((t-\tau)\cos\theta + \sigma\sin\theta\bigr)(t-\tau)}{\sigma^{2}+(t-\tau)^{2}}, \\
B	&\coloneqq \bigl((t-\tau)\cos\theta + \sigma\sin\theta\bigr)\biggl((1+\delta)\log\tau + \frac{\sigma}{\sigma^{2}+(t-\tau)^{2}}\biggr) + \frac{\bigl(\sigma\cos\theta -(t-\tau)\sin\theta\bigr)(t-\tau)}{\sigma^{2}+(t-\tau)^{2}}.
\end{align*}
Whence we see that (with $\gamma_{m}(y) = \sigma+\I t$)
\[
	\arg(\e^{-\I\pi/2}\gamma'_{m}(y)) = \arctan\biggl(\frac{\log x \cdot 2\frac{\sigma^{2}+(t-\tau)^{2}}{\tau^{1-(1+\delta)\sigma}} + A}{B}\biggr).
\]
The relation \eqref{eq: steepest path m} holds on the path of steepest descent, so
\[
	\arg(\e^{-\I\pi/2}\gamma'_{m}(y)) = \arctan\biggl(\frac{\frac{(t-\tau)\cos\theta + \sigma\sin\theta}{(t-t_{m})- v_{m}/\log x} + A}{B}\biggr).
\]
Now by \eqref{eq: bound v_{m}}, $v_{m}/\log x \ll (\log\tau)^{-9/4}$ and $t-t_{m} = \theta/((1+\delta)\log\tau) + O(m(\log\tau)^{-2})$, so that, for $\eta/2<\abs{\theta}\leq \pi /2$,
\begin{align*}
	\frac{(t-\tau)\cos\theta+\sigma\sin\theta}{(t-t_{m})-v_{m}/\log x} &= \frac{(1+\delta)\log\tau}{\theta}\bigl((t-\tau)\cos\theta+\sigma\sin\theta\bigr) + O_{\eta}((\log\tau)^{3/4})
	\\
	&
	=\frac{\sigma (1+\delta) \log\tau\sin\theta}{\theta} + O_{\eta}((\log\tau)^{3/4}).
\end{align*}
Furthermore, we have
\[
A= -\sigma(1+\delta)\log\tau\cos\theta + O((\log\tau)^{3/4}) \qquad \mbox{and} \qquad B= \sigma(1+\delta)\log\tau\sin\theta + O((\log\tau)^{3/4}),
\]
so 
\begin{align*}
	\arg(\e^{-\I\pi/2}\gamma'_{m}(y)) 	&= \arctan\biggl(\frac{\frac{\sin\theta}{\theta} - \cos\theta + O_{\eta}((\log\tau)^{-1/4})}{\sin\theta + O((\log\tau)^{-1/4})}\biggr) \\
								&= \arctan\bigl( 1/\theta - \cot\theta + O_{\eta}((\log\tau)^{-1/4})\bigr).
\end{align*}
Now, $1/\theta-\cot\theta$ is bounded on $[-\pi/2, \pi/2]$ with $\abs{1/\theta-\cot\theta} \le 2/\pi$, so we see that\footnote{$\arctan(2/\pi) \approx 0.18\pi$.} 
\[
	|\arg(\e^{-\I\pi/2}\gamma'_{m}(y))| \le \pi/5
\]
on the range $\eta/2<\abs{\theta}\leq \pi /2$, and combining this with the estimate above, we see that this inequality holds in the entire range $\abs{\theta}\le\pi/2$.

\subsection{The contribution from the saddle points}
 \label{subsec: The contribution from the saddle points} We will now estimate the contribution to the Perron integral in \eqref{eq: Perron formula} coming from the integrals over the paths of steepest descent $\Gamma_m$. Since we take the imaginary part of this Perron integral, we need to control the argument of $\int_{\Gamma_{m}}$, and see that this is close to $\pi/2$, or at least sufficiently far from 0 and $\pi$. For this we use the following lemma.
\begin{lemma}
\label{lem: estimation integral}
Let $a<b$ and suppose that $g\colon [a,b] \to \C$ is integrable. If there exist $\theta_{0}$ and $\eta$ with $0\le\eta<\pi/2$ such that $\abs{\arg(g\e^{-\I\theta_{0}})}\le\eta$, then 
\[	
	\int_{a}^{b} g(y)\dif y = \rho \e^{\I (\theta_{0}+\vphi)}
\]
for some real numbers $\rho$ and $\vphi$ satisfying 
\[
	\rho \ge (\cos\eta) \int_{a}^{b}\abs{g(y)}\dif y \quad\mbox{and}\quad \abs{\vphi}\le\eta.
\]
\end{lemma}

\begin{proof}
Assume that $g$ is not identically zero (that case is trivial) and write $g(y) = R(y)\e^{\I\theta(y)}$ with $\abs{\theta(y)-\theta_{0}}\le\eta$. Then,
\[
	\int_{a}^{b}g(y) \dif y = \e^{\I\theta_{0}}\biggl(\int_{a}^{b}R(y)\cos\bigl(\theta(y)-\theta_{0}\bigr)\dif y + \I \int_{a}^{b}R(y)\sin\bigl(\theta(y)-\theta_{0}\bigr)\dif y\biggr).
\]
The modulus of this expression is larger than 
\[
	\int_{a}^{b}R(y)\cos\eta\dif y,
\]
while 
\[
	\abs{\vphi} = \arctan\abs{\frac{\int_{a}^{b} R(y)\sin\bigl(\theta(y)-\theta_{0}\bigr)\dif y}{\int_{a}^{b} R(y)\cos\bigl(\theta(y)-\theta_{0}\bigr)\dif y}} \le \arctan\frac{\sin \eta}{\cos\eta} = \eta.
\]
\end{proof}
We have that 
\[
	\int_{\Gamma_{m}}\e^{f(s)}\frac{\exp\Bigl(\sideset{}{'}\sum\dotso\Bigr)}{(s-1)(s+1)} \dif s = \e^{f(s_{m})} \int_{y_{m}^{-}}^{y_{m}^{+}}\e^{f(\gamma_{m}(y)) - f(s_{m})}\frac{\exp\Bigl(\sideset{}{'}\sum\dotso\Bigr)}{(\gamma_{m}(y)-1)(\gamma_{m}(y)+1)}\gamma'_{m}(y)\dif y.
\]

Let us first focus on the argument of this integral. Using \eqref{eq: saddle point identity}, 
\[
	f(s_{m}) = (1+\sigma_{m})\log x + \I t_{m}\log x + \frac{\log x}{(1+\delta)\log\tau + \frac{1}{s_{m}-\I\tau}},
\]
whence, in view of \eqref{eq: approx t} and the bound on $v_{m}$ \eqref{eq: bound v_{m}}, we obtain
\[
	\Im f(s_{m}) = \log x\biggl( \tau + \frac{2\pi m}{(1+\delta)\log\tau}\biggl(1 - \frac{1+\sqrt{2}\sqrt{\frac{\log\log x}{\log x}}}{(1+\delta)\log\tau}\biggr)\biggr) 
	+ O\biggl(\frac{|m|\log x}{(\log\tau)^{3}} + \frac{|m|^{3}\log x}{(\log\tau)^{4}}\biggr). 
\]
By the technical assumptions imposed on the sequence $(\tau_{j})_{j}$, more specifically properties \ref{Property (c)} and \ref{Property (d)}, we see that the main term of the above expression has distance at most $\pi/16$ from an even multiple of $\pi$ when $k$ is even or from an odd multiple of $\pi$ when $k$ is odd. Furthermore, by restricting the range for $m$ to $\abs{m} \le c(\log x)^{1/3}(\log\log x)^{2/3}$ for some sufficiently small absolute constant $c>0$, we force the error term to be at most $\pi/16$. Hence, for $\abs{m} \le c(\log x)^{1/3}(\log\log x)^{2/3}$,
\begin{equation*}
d\bigl(\Im f(s_m), 2\Z\pi\bigr)<  \frac{\pi}{8} \text{ for } k \text{ even} \qquad\text{ and }\qquad   d\bigl(\Im f(s_m), \pi+ 2\Z\pi\bigr) < \frac{\pi}{8} \text{ for } k \text{ odd}.
\end{equation*}
Second, in the range $\abs{t-\tau_{k}}\le 1$, $1/2\le\sigma\le1$ say, we have by property \ref{Property (a)} that
\[
	\sideset{}{'}\sum = \frac{1}{2}\sum_{j\neq k}\tau_{j}\frac{\tau_{j}^{-(1+\delta_{j})s}-\tau_{j}^{-\nu_{j}s}}{s-\I\tau_{j}}+\frac{1}{2}\sum_{j=0}^{\infty}\tau_{j}\frac{\tau_{j}^{-(1+\delta_{j})s}-\tau_{j}^{-\nu_{j}s}}{s+\I\tau_{j}} \ll \sum_{j=0}^{\infty}\tau_{j}^{-1/2}.
\]
We can make this smaller than $\pi/16$ by choosing $\tau_{0}$ sufficiently large.
Third, in the same range we have
\[
	\frac{1}{\abs{(s-1)(s+1)}} 			= \frac{1}{\tau_{k}^{2}} + O\biggl(\frac{1}{\tau_{k}^{3}}\biggr), \quad
	\arg\biggl(\frac{1}{(s-1)(s+1)}\biggr)	= \pi + O\biggl(\frac{1}{\tau_{k}}\biggr). 					
\]
Finally, by the results of the previous subsection, $\abs[0]{\arg(\e^{-\I\pi/2}\gamma'_{m}(y))} < \pi/5$. Combining all of this, and using the fact that on $\Gamma_{m}$, $f(s)-f(s_{m})$ is real, we get by Lemma \ref{lem: estimation integral} that 
\[
	\int_{\Gamma_{m}}\e^{f(s)}\frac{\exp\Bigl(\sideset{}{'}\sum\dotso\Bigr)}{(s-1)(s+1)} \dif s = (-1)^{k+1}R_{m}\e^{\I(\pi/2 + \vphi_{m})},
\]
where $\abs{\vphi_{m}} < 2\pi/5$ and $R_{m}$ is a positive number satisfying
\[
	R_{m} \gg \frac{\e^{\Re f(s_{m})}}{\tau^{2}}\int_{y_{m}^{-}}^{y_{m}^{+}}\exp\bigl(\Re (f(\gamma_{m}(y)) - f(s_{m}))\bigr)\dif y.
\]

We see that the imaginary part of the integral over $\Gamma_{m}$ always has sign $(-1)^{k+1}$, so the sum over $m$ of the imaginary parts of all these integrals has also sign $(-1)^{k+1}$ and is in absolute value larger than the contribution of the integral over $\Gamma_{0}$. Therefore, the saddle point contribution of the integrals $\Gamma_m$, $m\neq0$, cannot destroy the contribution of $\Gamma_0$, and as we have better control over  $\sigma_0$ compared to general $\sigma_m$ (\eqref{eq: approx sigma_{0}} versus \eqref{eq: approx sigma}), we will proceed to get a lower bound for the contribution of the saddle points    only\footnote{With some more calculations, one can show that the size of the contributions of the saddle points $s_{m}$ to the Perron integral decays relatively slowly with respect to $m$. Therefore, if the phases of the contributions were arbitrary, the contribution of the saddle point $s_{0}$ might get cancelled by the others. As a consequence we were led to this intricate analysis of the phases of the contributions.} using the integral $\Gamma_0$.
To estimate $R_{0}$, we restrict the integration interval to the interval which corresponds to the part of $\Gamma_{0}$ which lies inside $B(s_{0}, \eta/\log\tau)$ for some $\eta>0$. By Lemma \ref{lem: approx f and f'} and \eqref{eq: f''}, we see that
\[
	\int_{y_{0}^{-}}^{y_{0}^{+}}\exp\bigl(\Re (f(\gamma_{0}(y)) - f(s_{0}))\bigr)\dif y \gg \frac{1}{\sqrt{\log x\log\tau}}.
\]
Now by \eqref{eq: saddle point identity}, we have
\begin{equation}
\label{eq: power of tau}
	\frac{1}{2}\frac{\tau^{1-(1+\delta)\sigma_{0}}}{\sigma_{0}} = \sqrt{2}\sqrt{\frac{\log x}{\log\log x}} + O(1),
\end{equation}
and by using \eqref{eq: approx sigma_{0}}, we see that
\[
	\frac{\e^{\Re f(s_{0})}}{\tau^{2}} = \exp\biggl( 2\log x - \sqrt{2}\sqrt{\log x\log\log x} - \sqrt{2}(a+\log 2-1)\sqrt{\frac{\log x}{\log\log x}} - 2\log \tau + O(\log\log x)\biggr).
\]

Hence, we can conclude that the contribution of saddle points with $|m| \leq c(\log x)^{1/3} (\log \log x)^{2/3}$ has sign $(-1)^{k+1}$ and is bounded in absolute value from below as follows:
\begin{align}
	&\abs{\frac{1}{\pi}\Im \sum_{m} \int_{\Gamma_{m}}\e^{f(s)}\frac{\exp\Bigl(\sideset{}{'}\sum\dotso\Bigr)}{(s-1)(s+1)} \dif s}\nonumber \\ 	
			&\gg \exp\biggl((1+\sigma_{0})\log x - \sqrt{2}\sqrt{\log x\log\log x} + \sqrt{2}\sqrt{\frac{\log x}{\log\log x}} + O(\log\log x)\biggr) \label{eq: contribution s_{0}}\\
			&= x^{2}\exp\biggl(-2\sqrt{2}\sqrt{\log x\log\log x} - \sqrt{2}(a+\log 2-1)\sqrt{\frac{\log x}{\log\log x}}  + O(\log\log x) \biggr). \label{eq: contribution saddle points}
\end{align}

\section{The remainder}
\label{sec: The remainder}
Recall that the Perron integral we are considering is given by 
\begin{equation}
\label{eq: Perron integral}
	\int \e^{f(s)}\frac{\exp\left(\sideset{}{'}\sum\dotso\right)}{(s-1)(s+1)}\dif s.	
\end{equation}
In this section we will show that one can integrate over a contour $\Gamma$ which incorporates the paths of steepest descent $\Gamma_{m}$, in such a way that the integral over $\Gamma\setminus\cup_{m}\Gamma_{m}$ is of strictly lower order than the contribution of the saddle points. For brevity, we will omit the integrand when writing integrals. The omitted integrand is always meant to be the integrand of \eqref{eq: Perron integral}. Let us begin by bounding the zeta function on the contour of Hilberdink and Lapidus.
\begin{lemma}
\label{lem: bound zeta on GHL}
For $t\ge \e^{\e}$, the function 
\begin{gather*}
	\sum_{j=0}^{\infty}\frac{\tau_{j}}{2}\biggl( \frac{\tau_{j}^{-(1+\delta_{j})s} - \tau_{j}^{-\nu_{j}s}}{s - \I\tau_{j}} + \frac{\tau_{j}^{-(1+\delta_{j})s} - \tau_{j}^{-\nu_{j}s}}{s + \I\tau_{j}}\biggr), \quad \text{resp. }\\
	\sum_{j\neq k}\frac{\tau_{j}}{2}\frac{\tau_{j}^{-(1+\delta_{j})s} - \tau_{j}^{-\nu_{j}s}}{s - \I\tau_{j}} + \sum_{j=0}^{\infty}\frac{\tau_{j}}{2}\frac{\tau_{j}^{-(1+\delta_{j})s} - \tau_{j}^{-\nu_{j}s}}{s + \I\tau_{j}},
\end{gather*}
is bounded in the region
\[
	\biggl\{ \, s=\sigma+\I t : \sigma \ge 1 - \frac{\log\log t}{\log t}\, \biggr\}, \quad \text{resp.\ }\quad  \{\, s=\sigma+\I t: \sigma\ge1/2, \tau_{k}^{1/5} \le t \le\tau_{k}^{5}\,\}.
\]
\end{lemma}
\begin{proof}
For $\sigma\ge1-\log\log t/\log t$, we have
\begin{align*}
\tau_{j}^{1-(1+\delta_{j})\sigma} 	&\le \exp\biggl( \log\tau_{j}\biggl( -\frac{\log\log\tau_{j} + a_{j}}{\log\tau_{j}} + \frac{\log\log t}{\log t} + \frac{\log\log\tau_{j} + a_{j}}{\log\tau_{j}}\frac{\log\log t}{\log t}\biggr)\biggr)\\
							&= \exp\biggl(-\log\log\tau_{j}-a_{j} + \frac{\log\log t}{\log t}\bigl( \log\tau_{j} + \log\log\tau_{j} + a_{j}\bigr)\biggr). 
\end{align*}
Since $\log\log t/\log t \le 1/2$, this is $\ll \sqrt{\tau_{j}/\log\tau_{j}}$. Furthermore, if $\tau_{j}/2\le t$, the above is 
\[
	\ll \exp\biggl(-\log\log\tau_{j} + \frac{\log(\log\tau_{j} - \log 2)}{\log\tau_{j}-\log2}\log\tau_{j} + O(1)\biggr) \ll 1.
\]
Also, for $\sigma\ge1/2$, $\tau_{j}^{1-(1+\delta_{j})\sigma} \le \tau_{j}^{2/3}$ say.
For the first case, we bound $\abs{s-\I\tau_{j}}$ from below by $1/2$ if $\tau_{j}/2\le t \le 2\tau_{j}$ and by $\tau_{j}/2$ otherwise, while for the second case, we bound $\abs{s-\I\tau_{j}}$, $j\neq k$, from below by $\tau_{j}/2$, in view of property \ref{Property (a)}. Hence, we see that the functions are 
\[
	\ll \Biggl(\sum_{j=0}^{\infty}\sqrt{\frac{1}{\tau_{j}\log\tau_{j}}}\, \Biggr) + 1, \quad \text{resp.} \quad \sum_{j=0}^{\infty}\frac{1}{\tau_{j}^{1/3}},
\]
which are bounded (by property \ref{Property (a)}).
\end{proof}
The lemma implies that we may indeed shift the contour from the line $\sigma = \kappa$ to the Hilberdink-Lapidus contour $\GHL$ as described in Section \ref{sec: Setup and overview}. Indeed, the integral over the segments 
\[
	\biggl[ 1-\frac{\log\log\abs{T}}{\log\abs{T}} \pm \I T, \kappa\pm\I T\biggr]
\]
is $O_{\kappa}(x^{\kappa+1}/T^{2})$, which tends to 0 as $T\to \infty$.

\subsection{Connecting the steepest paths}
Set $\mmax \coloneqq \lfloor c(\log x)^{1/3}(\log\log x)^{2/3}\rfloor$, and 
\[
	T_{1}^{\pm} \coloneqq \tau \pm \frac{2\pi \mmax + \pi/2}{(1+\delta)\log\tau} = t_{\pm\mmax}^{\pm}.
\]
In this subsection we will show that one can connect the different paths of steepest descent $\Gamma_{m}$ to form one contour whose imaginary part ranges from $T_{1}^{-}$ to $T_{1}^{+}$.

For $m$ in the range $\abs{m} \le c(\log x)^{1/3}(\log\log x)^{2/3}$, set $\sigma_{m}^{\pm} \coloneqq \Re \gamma_{m}(y_{m}^{\pm})$. By \eqref{eq: steepest path m} (the value of $\theta$ is here either $\pi/2$ or $-\pi/2$ correspondingly, cf. \eqref{eq: t+-m}), we see that these numbers satisfy
\[
	\sigma_{m}^{\pm} = \frac{1}{1+\delta}
	\biggl(1 - \frac{1}{\log\tau}\biggl(\log\log x + \log 2 +  \log\bigl((\sigma_{m}^{\pm})^{2} + (t_{m}^{\pm}-\tau)^{2}\bigr) + \log\biggl(\frac{(t_{m}^{\pm}-t_{m}) - v_{m}/\log x}{\pm\sigma_{m}^{\pm}}\biggr)\biggr)\biggr).
\]
Using that 
\[
	1/2<\sigma_{m}^{\pm}<1, \quad t_{m}^{\pm}-t_{m} = \pm\frac{\pi}{2}\frac{1}{(1+\delta)\log\tau} + O\biggl(\frac{(\log\log\tau)^{1/3}}{(\log\tau)^{4/3}}\biggr), \quad \frac{v_{m}}{\log x} \ll \frac{1}{(\log\tau)^{9/4}}, 
\]
we get 
\begin{equation}\label{eq: sigma_{m}^{pm}}
	\sigma_{m}^{\pm} 	= 1-\sqrt{2}\sqrt{\frac{\log\log x}{\log x}} - \frac{\sqrt{2}(a+\log 2+ \log \pi/2)}{\sqrt{\log x\log\log x}} + O\biggl(\frac{1}{(\log x)^{2/3}(\log\log x)^{1/3}}\biggr).
	\end{equation}
Consider now the contour $\Upsilon_{m}$ which connects $\sigma_{m}^{+}+\I t_{m}^{+}$ with $\sigma_{m+1}^{-}+\I t_{m+1}^{-}$ via a vertical and horizontal line: $\Upsilon_{m} \coloneqq [\sigma_{m}^{+}+\I t_{m}^{+}, \sigma_{m}^{+}+\I t_{m+1}^{-}] \cup [\sigma_{m}^{+}+\I t_{m+1}^{-},\sigma_{m+1}^{-}+\I t_{m+1}^{-}]$.
We have 
\[
	\Re \frac{1}{2}\frac{\tau^{1-(1+\delta)s}}{s-\I\tau} = \frac{1}{2}\frac{\tau^{1-(1+\delta)\sigma}}{\sigma^{2}+(t-\tau)^{2}}\bigl( \sigma\cos((t-\tau)(1+\delta)\log\tau) - (t-\tau)\sin((t-\tau)(1+\delta)\log\tau)\bigr).
\]
On this contour,
\[
	\sigma\cos\bigl((t-\tau)(1+\delta)\log\tau\bigr) - (t-\tau)\sin\bigl((t-\tau)(1+\delta)\log\tau\bigr) \le t_{\abs{m}+1}^{-}-\tau \ll \biggl(\frac{\log\log\tau}{\log\tau}\biggr)^{1/3},
\]
since the argument of $\cos$ belongs to $ [\pi/2, 3\pi/2] + 2\pi\Z$. Therefore, by \eqref{eq: sigma_{m}^{pm}},
\begin{align*}
	\Re \frac{1}{2}\frac{\tau^{1-(1+\delta)s}}{s-\I\tau} &\le \frac{1}{2}\frac{\tau^{1-(1+\delta)\sigma}}{\sigma^{2}+(t-\tau)^{2}}(t^{-}_{\abs{m}+1}-\tau) \\
		&\ll \tau^{1-(1+\delta)\sigma}\biggl(\frac{\log\log x}{\log x}\biggr)^{1/6} \ll \biggl(\frac{\log x}{\log\log x}\biggr)^{1/3}.
\end{align*}
Hence, in view of \eqref{eq: sigma_{m}^{pm}} and using Lemma \ref{lem: bound zeta on GHL} to bound the series $\sideset{}{'}\sum$, we see that the integrand of \eqref{eq: Perron integral} over $\Upsilon_{m}$ is
\[
	\ll x^{2}\exp\Biggl( -2\sqrt{2}\sqrt{\log x\log\log x} - \sqrt{2}\bigl(a+\log2+\log\frac{\pi}{2}\bigr)\sqrt{\frac{\log x}{\log\log x}} + O\biggl(\biggl(\frac{\log x}{\log\log x}\biggr)^{1/3}\biggr)\Biggr).
\]
Summing over $\abs{m}\le c(\log x)^{1/3}(\log\log x)^{2/3}$, we get that $\int_{\cup_{m}\Upsilon_{m}}$ is of lower order than the main contribution given by \eqref{eq: contribution saddle points}. 

\subsection{Returning to the original contour}
Finally we show that we can connect the contour $\cup_{m}\bigl(\Gamma_{m}\cup\Upsilon_{m}\bigr)$ to $\GHL$  (see Figure 2).
 First we go from the endpoints of $\Gamma_{\pm\mmax}$ to the line $\sigma=\sigma_{0}$: set 
\[
	\Delta_{0}^{-}\coloneqq [\sigma_{-\mmax}^{-}+\I T_{1}^{-}, \sigma_{0}+ \I T_{1}^{-}], \quad \Delta_{0}^{+} \coloneqq [\sigma_{\mmax}^{+}+ \I T_{1}^{+}, \sigma_{0}+\I T_{1}^{+}].
\]
On this contour, we have 
\[
	\Re \frac{1}{2}\frac{\tau^{1-(1+\delta)s}}{s-\I\tau} < 0,
\]
which together with Lemma \ref{lem: bound zeta on GHL} implies that $\int_{\Delta_{0}^{\pm}} \ll \exp\bigl((1+\sigma_{0})\log x - \sqrt{2}\sqrt{\log x\log\log x}\bigr)$, which is admissible in view of \eqref{eq: contribution s_{0}}.

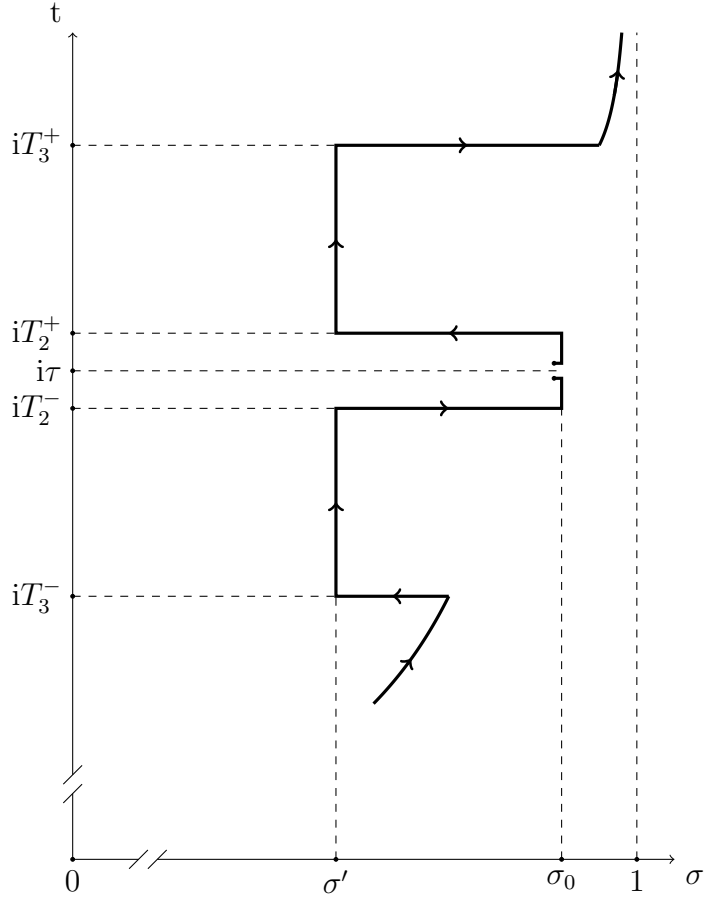
\begin{figure}
\begin{tikzpicture}[scale=0.5] 
	\draw[-] (0,-2) -- (0,-0.25);
	\draw[-] (-0.25,-0.5) -- (0.25,0);
	\draw[-] (-0.25,0) -- (0.25,0.5);
	\draw[->] (0,0.25) -- (0, 20);
	\node[above left] at (0,20) {t};
	\draw[-] (0,-2) -- (1.75,-2);
	\draw[-] (1.5,-2.25) -- (2, -1.75);
	\draw[-] (2,-2.25) -- (2.5,-1.75);
	\draw[->] (2.25,-2) -- (16,-2);
	\node[below right] at (16,-2) {$\sigma$};
	\draw[fill] (0,-2) circle [radius=0.05];
	\node[below] at (0,-2) {$0$};
	\draw[fill] (15,-2) circle [radius=0.05];
	\node[below] at (15,-2) {$1$};
	\draw[fill] (0,5) circle [radius=0.05];
	\node[left] at (0,5) {$\I T_{3}^{-}$};
	\draw[fill] (0,10) circle [radius=0.05];
	\node[left] at (0,10) {$\I T_{2}^{-}$};
	\draw[fill] (0,11) circle [radius=0.05];
	\node[left] at (0,11) {$\I \tau$};
	\draw[fill] (0,12) circle [radius=0.05];
	\node[left] at (0,12) {$\I T_{2}^{+}$};
	\draw[fill] (0,17) circle [radius=0.05];
	\node[left] at (0,17) {$\I T_{3}^{+}$};
	\draw[fill] (13,-2) circle [radius=0.05];
	\node[below] at (13,-2) {$\sigma_{0}$};
	\draw[fill] (7,-2) circle [radius=0.05];
	\node[below] at (7,-2) {$\sigma'$};
	\draw[fill] (12.8, 10.8) circle [radius=0.05];
	\draw[fill] (12.8, 11.2) circle [radius=0.05];
	
	\draw[dashed] (7,-2) -- (7,5);
	\draw[dashed] (13,-2) -- (13,10);
	\draw[dashed] (15,-2) -- (15,20);
	\draw[dashed] (0,5) -- (7,5);
	\draw[dashed] (0,10) -- (7,10);
	\draw[dashed] (0,11) -- (13, 11);
	\draw[dashed] (0,12) -- (7,12);
	\draw[dashed] (0,17) -- (7,17);
		
	\draw[very thick, ->, domain=8:9] plot ({\x}, {50/(15-\x)-5});
	\draw[very thick, domain=8.9:10] plot ({\x}, {50/(15-\x)-5});
	\draw[very thick, ->] (10,5) -- (8.5,5);
	\draw[very thick, ->] (8.6,5) -- (7,5) -- (7,7.5);
	\draw[very thick, ->] (7,7.4) -- (7,10) -- (10,10);
	\draw[very thick] (9.9,10) -- (13,10) -- (13,10.8) -- (12.8,10.8);
	\draw[very thick, ->] (12.8,11.2) -- (13,11.2) -- (13,12) -- (10,12);
	\draw[very thick, ->] (10.1,12) -- (7,12) -- (7,14.5);
	\draw[very thick, ->] (7,14.4) -- (7,17) -- (10.5,17);
	\draw[very thick] (10.4,17) -- (14,17);
	\draw[very thick, ->, domain=14:14.5] plot({\x}, {2/(15-\x)+15});
	\draw[very thick, domain=14.4:14.6] plot({\x}, {2/(15-\x)+15});
\end{tikzpicture}
\caption{Returning to the original contour of Hilberdink and Lapidus.}
\end{figure}
	
Suppose that 
\[
	\abs{t-\tau} \ge \frac{2\pi \mmax}{(1+\delta)\log\tau} \ge c\pi\sqrt{2}\biggl(\frac{\log\log x}{\log x}\biggr)^{1/6}.
\]
Then 
\begin{align*}
	\frac{1}{2}\frac{\tau^{1-(1+\delta)\sigma_{0}}}{\abs{\sigma_{0}+\I(t-\tau)}} 	&= \frac{1}{2}\frac{\tau^{1-(1+\delta)\sigma_{0}}}{\sigma_{0}}\frac{1}{\sqrt{1+\bigl((t-\tau)/\sigma_{0}\bigr)^{2}}} \\
															&\le \frac{1}{2}\frac{\tau^{1-(1+\delta)\sigma_{0}}}{\sigma_{0}}\biggl(1 - c'\biggl(\frac{\log\log x}{\log x}\biggr)^{1/3}\biggr),
\end{align*}
for some constant $c'>0$. By \eqref{eq: power of tau}, we thus have
\[
	\frac{1}{2}\frac{\tau^{1-(1+\delta)\sigma_{0}}}{\abs{\sigma_{0}+\I(t-\tau)}} \le \sqrt{2}\sqrt{\frac{\log x}{\log\log x}} - c''\biggl(\frac{\log x}{\log\log x}\biggr)^{1/6} + O(1),
\]
for some constant $c''>0$. We now set 
\begin{equation}
\label{eq: definition T_{2} with c''} 
	T_{2}^{\pm} \coloneqq \tau \pm \exp\biggl(\frac{c''}{2}\biggl(\frac{\log x}{\log\log x}\biggr)^{1/6}\biggr), \quad
	\Delta_{1}^{\pm} \coloneqq [\sigma_{0} + \I T_{1}^{\pm}, \sigma_{0} + \I T_{2}^{\pm}].
\end{equation}
On $\Delta_{1}^{\pm}$, $t\asymp \tau$ and the series $\sideset{}{'}\sum$ is bounded by Lemma \ref{lem: bound zeta on GHL}. Therefore, we have
\begin{multline*}
	\int_{\Delta_{1}^{\pm}} \ll \exp\biggl(\frac{c''}{2}\biggl(\frac{\log x}{\log\log x}\biggr)^{1/6}\biggr)x^{1+\sigma_{0}} \times \\
	\exp\biggl(-\sqrt{2}\sqrt{\log x\log\log x} + \sqrt{2}\sqrt{\frac{\log x}{\log\log x}} - c''\biggl(\frac{\log x}{\log\log x}\biggr)^{1/6} + O(1)\biggr),
\end{multline*}
which is of lower order than the contribution of the saddle points \eqref{eq: contribution s_{0}}.

Next, set 
\begin{align*}
	\sigma' 	&\coloneqq \frac{1}{1+\delta}\biggl(1-\frac{\sqrt{2}}{2}\frac{c''}{(\log x)^{1/3}(\log\log x)^{2/3}}\biggr) \\
			&= 1 - \frac{\sqrt{2}}{2}\frac{c''}{(\log x)^{1/3}(\log\log x)^{2/3}} + O\biggl( \sqrt{\frac{\log\log x}{\log x}}\biggr).
\end{align*}
We point out that $\sigma'<\sigma_{0}$ (for $x$ sufficiently large).
Then, for $t\ge T_{2}^{+}$ or $t\le T_{2}^{-}$, 
\begin{equation}
\label{eq: bound tau^{...} away from tau}
	\frac{1}{2}\frac{\tau^{1-(1+\delta)\sigma'}}{\abs{\sigma' + \I(t-\tau)}} \ll \exp\biggl( -\frac{c''}{2}\biggl(\frac{\log x}{\log\log x}\biggr)^{1/6} + \frac{c''}{2}\biggl(\frac{\log x}{\log\log x}\biggr)^{1/6}\biggr) = 1.
\end{equation}
Set $\Delta_{2}^{\pm} \coloneqq [\sigma_{0}+ \I T_{2}^{\pm}, \sigma' + \I T_{2}^{\pm}]$. We obtain that $\int_{\Delta_{2}^{\pm}} \ll x^{1+\sigma_{0}}\exp\bigl(-\sqrt{2}\sqrt{\log x\log\log x}\bigr)$,
which is again admissible, by comparing with \eqref{eq: contribution s_{0}}.

Consider now $T_{3}^{+} \coloneqq \tau^{5}$, $T_{3}^{-}\coloneqq \tau^{1/5}$, and $\Delta_{3}^{\pm} \coloneqq [\sigma' + \I T_{2}^{\pm}, \sigma' + \I T_{3}^{\pm}]$. We get (again bounding $\sideset{}{'}\sum$ using Lemma \ref{lem: bound zeta on GHL} and using \eqref{eq: bound tau^{...} away from tau})
\[
	\int_{\Delta_{3}^{\pm}} \ll x^{2}\exp\biggl(-\frac{\sqrt{2}}{2}c''\biggl(\frac{\log x}{\log\log x}\biggr)^{2/3} + O(\sqrt{\log x\log\log x})\biggr),
\]
which is negligible.

Next we move to the contour $\GHL$: set $\Delta_{4}^{\pm} \coloneqq [\sigma' + \I T_{3}^{\pm}, 1-\log\log T_{3}^{\pm}/\log T_{3}^{\pm} + \I T_{3}^{\pm}]$. Now 
\[
	1 - \frac{\log\log \tau^{1/5}}{\log\tau^{1/5}} \le 1- \frac{9}{4}\sqrt{2}\sqrt{\frac{\log\log x}{\log x}}
\]
say, so again by \eqref{eq: bound tau^{...} away from tau} and Lemma \ref{lem: bound zeta on GHL} $\int_{\Delta_{4}^{-}}\ll x^{2}\exp\bigl(-\frac{9}{4}\sqrt{2}\sqrt{\log x\log\log x}\bigr)$, which is admissible with respect to \eqref{eq: contribution saddle points}. Also,
\[
	\frac{1}{\tau^{5}} = \exp\bigl(-\frac{5\sqrt{2}}{2}\sqrt{\log x\log\log x}\bigr),
\]
so $\int_{\Delta_{4}^{+}}$ is admissible as well.

Finally, set $\GHL^{-}\coloneqq \{\, s\in \GHL: t\le \tau^{1/5}\, \}$, $\GHL^{+}\coloneqq \{\, s\in \GHL: t\ge \tau^{5}\,\}$. By Lemma \ref{lem: bound zeta on GHL}, the series $\sum_{k}$ is bounded on these contours. We get\footnote{The part of the integral for $t<\e^{\e}$ is $\ll x^{2-1/\e}$.} 
\begin{align*}
	\int_{\GHL^{-}} 	&\ll x^{2}\int_{\e}^{(\log\tau)/5}\exp\biggl(-\frac{\log u}{u}\log x - u\biggr)\dif u \ll x^{2}\exp\biggl(-\frac{\log\log\tau^{1/5}}{\log\tau^{1/5}}\log x\biggr) \\
				&\ll x^{2}\exp\biggl(-\frac{9}{4}\sqrt{2}\sqrt{\log x\log\log x}\biggr),\\
	\int_{\GHL^{+}} 	&\ll x^{2}\int_{5\log\tau}^{\infty}\exp\biggl(-\frac{\log u}{u}\log x - u\biggr)\dif u \ll x^{2}\exp(-5\log\tau) \\
				&\ll x^{2}\exp\biggl(-\frac{5}{2}\sqrt{2}\sqrt{\log x\log\log x}\biggr),
\end{align*}
and these integrals are therefore also negligible.

\section{Conclusion of analysis of continuous example: proof of Lemma \ref{lem: technical properties sequences}} \label{section: conclusion proof lemma}
The results of the previous two sections now yield the relation \eqref{eq: omega for int N}, which as we have already remarked at the end of Section \ref{sec: Setup and overview} suffices to conclude the proof of Theorem \ref{th: optimality HL}. All of this only remains true provided that Lemma \ref{lem: technical properties sequences} holds. Let us now prove Lemma \ref{lem: technical properties sequences}. 

\begin{proof}[Proof of Lemma \ref{lem: technical properties sequences}]
We will show that for any $T$ one can find $\tau_{k}$,  $a_{k}$, and $\nu_{k}$ which obey $\tau_{k}>T$ and the properties \ref{Property (b)}-\ref{Property (d)}. This will allow us to inductively define the sequences for which also \ref{Property (a)}  holds, by selecting $T\coloneqq (2\tau_{k})^{5}$ in step $k+1$.

Set $\alpha\coloneqq \log 6 + 1/2$, and pick a number $\xi_{k} > \e^{T^{2}}$ such that 
\begin{multline*}
	\frac{\log\xi_{k}}{\frac{1}{\sqrt{2}}\sqrt{\log\xi_{k}\log\log\xi_{k}} + \frac{1}{2}\log\log\xi_{k} + \frac{1}{2}\log\log\log\xi_{k} + \alpha -\frac{\log2}{2}} \times \\
		\Biggl(1-\frac{1+\sqrt{2}\sqrt{\frac{\log\log\xi_{k}}{\log\xi_{k}}}}{\frac{1}{\sqrt{2}}\sqrt{\log\xi_{k}\log\log\xi_{k}} + \frac{1}{2}\log\log\xi_{k} + \frac{1}{2}\log\log\log\xi_{k} + \alpha -\frac{\log2}{2}}\Biggr) \in \Z
\end{multline*}
(corresponding to property \ref{Property (d)}). This is possible since this function is continuous and tends to $\infty$ as $\xi_{k}\to\infty$. Next, define $x_{k}\coloneqq \exp(\log\xi_{k}+\eps_k)$, with $\eps_k$ the smallest positive number such that 
\[
	\log x_{k}\exp\biggl(\frac{1}{\sqrt{2}}\sqrt{\log x_{k}\log\log x_{k}}\biggr) \in \begin{cases} 2\pi\Z \qquad & \mbox{ if } k \mbox{ is even}
	\\
\pi +2\pi\Z \qquad & \mbox{ if } k \mbox{ is odd}.
	\end{cases}
\]
Applying the mean value theorem to the function $y\exp\bigl(\frac{1}{\sqrt{2}}\sqrt{y\log y}\bigr)$, we see that
\[
	\eps_k \ll \frac{1}{\sqrt{\log\xi_{k}\log\log\xi_{k}}\exp\bigl(\frac{1}{\sqrt{2}}\sqrt{\log\xi_{k}\log\log\xi_{k}}\bigr)}.
\]
Hence, by replacing $\log\xi_{k}$ by $\log x_{k}= \log\xi_{k} + \eps_k$, 
we introduce an error in the condition for \ref{Property (d)} of order
\begin{align*}
	\ll \eps_k \biggl(\sqrt{\frac{y}{\log y}}\biggr)'\biggr\rvert_{y=\log\xi_{k}} 	&\ll \frac{1}{\log\xi_{k}\log\log\xi_{k}\exp\bigl(\frac{1}{\sqrt{2}}\sqrt{\log\xi_{k}\log\log\xi_{k}}\bigr)} \\
												&\asymp  \frac{1}{\log x_{k}\log\log x_{k}\exp\bigl(\frac{1}{\sqrt{2}}\sqrt{\log x_{k}\log\log x_{k}}\bigr)},
\end{align*}
which is admissible.

Next, set $a_{k}=\alpha+\eta_k$, with $\eta_k$ the smallest positive number such that 
\[
	\exp\biggl(\sqrt{\frac{\log x_{k}\log\log x_{k}}{2}}\biggr)\biggl(\sqrt{\frac{\log x_{k}\log\log x_{k}}{2}}+ \frac{1}{2}\log\log x_{k} + \frac{1}{2}\log\log\log x_{k}-\frac{\log2}{2}+\alpha+\eta_k\biggr)
\]	
belongs to $2\pi\Z$, corresponding to the first requirement of property \ref{Property (b)}. Then
\[
	\eta_k \ll \exp\biggl(-\frac{1}{\sqrt{2}}\sqrt{\log x_{k}\log\log x_{k}}\biggr),
\]
and replacing $\alpha$ by $a_{k} = \alpha+\eta_k$, 
the newly introduced error in the condition
 for \ref{Property (d)} is
\[
	\ll \frac{1}{\log\log x_{k}\exp\bigl(\frac{1}{\sqrt{2}}\sqrt{\log x_{k}\log\log x_{k}}\bigr)},
\]
which is admissible. Finally, we set $\tau_{k} = \exp\bigl(\frac{1}{\sqrt{2}}\sqrt{\log x_{k}\log\log x_{k}}\bigr)$, and we can choose a value for $\nu_{k}$ between 2 and 3 to satisfy the second part of property \ref{Property (b)}. All the properties are now fulfilled.
\end{proof}

\section{Discretization: proof of Theorem \ref{th: discrete optimality HL} }
\label{section discretization}
This last section is devoted to completing the proof of Theorem \ref{th: discrete optimality HL}. We will apply the probabilistic approach of Diamond, Montgomery, and Vorhauer \cite{DiamondMontgomeryVorhauer} and complement it with a new procedure of adding finitely many well-chosen primes. This allows us to obtain a suitable random approximation to the continuous prime measure $\dif \Pi_{C}$ that we have been studying in the previous sections and will enable us to select a random discrete Beurling prime number system having the desired  properties  \eqref{eq: discrete prime approximation} and \eqref{eq: discrete integer approximation}.

Let $1=v_{0}<v_{1}<v_{2}<\dotso$ be a fixed sequence of real numbers tending slowly to $\infty$ (how slowly will be specified later), and set 
\[
	q_{j} \coloneqq \int_{v_{j-1}}^{v_{j}} \dif \Pi_{C}(u).
\]
We will include the number $v_{j}$ as a prime in our discrete prime number system with probability $q_{j}$, where  as our first requirement on the $v_j$ we ask that they increase sufficiently slowly such that all $q_j < 1$. To make this precise, let $(X_{j})_{j>0}$ be a sequence of independent Bernoulli variables with parameters $q_{j}$ of success on a fixed probability space. Given a point $\omega$ in the probability space, denote by $\MP(\omega)$ the set of those $v_{j}$ for which $X_{j}(\omega) = 1$. The idea is to show that the probability that a prime number system $\MP(\omega)$ satisfies the bounds we need is nonzero.

Denote
\[
	S(y;t)=S_{\omega}(y;t) \coloneqq \int_{1}^{y^{+}}u^{-\I t}\dif\pi(u), \quad S_{C}(y;t) \coloneqq \int_{1}^{y}u^{-\I t}\dif \Pi_{C}(u),
\]
where $\pi=\pi_{\omega}$ is the prime counting function of the discrete system $\MP(\omega)$. We require the following bounds:
\begin{enumerate}[label = (\Alph*)]
	\item $\pi(y) = \Pi_{C}(y) + O(\sqrt{y})$ for $y$ sufficiently large;\label{Bound (A)}
	\item $S(y;t) = S_{C}(y;t) + O(\sqrt{y\log\abs{t}})$ uniformly for $y$ and $\abs{t}$ sufficiently large;\label{Bound (B)}
	\item $S(y;t) = S_{C}(y;t) + O(\sqrt{y}(\log\tau_{k})^{1/4})$ uniformly in the range 
	\[ 
		\abs{t-\tau_{k}} \le \exp\biggl(\frac{c''}{2}\biggl(\frac{\log x_{k}}{\log\log x_{k}}\biggr)^{1/6}\biggr),
	\]
	for $y$ and $k$ sufficiently large. The constant $c''$ is the same constant as the one appearing in \eqref{eq: definition T_{2} with c''}. \label{Bound (C)}
\end{enumerate}

\begin{lemma}
\label{lem: probabilistic lemma}
For every $\eps>0$, there exist $Y_{\eps}$, $T_{\eps}$, and $K_{\eps}$ such that the probability that a prime number system $\MP(\omega)$ satisfies the bounds \ref{Bound (A)}-\ref{Bound (C)} with $y\ge Y_{\eps}$, $\abs{t}\ge T_{\eps}$, and $k\ge K_{\eps}$ is at least $1-\eps$.
\end{lemma}
We will only prove the lemma for the bounds \ref{Bound (C)}. The proof of the validity of the bounds \ref{Bound (A)} and \ref{Bound (B)} is identical to that of \cite[Lemma 9]{DiamondMontgomeryVorhauer}. In fact, the only assumption that is needed in the proof of \cite[Lemma 9]{DiamondMontgomeryVorhauer} is that the measure $\dif \Pi_{C}$ satisfies
\[ 
	\frac{\dif u}{\log (2u)} \ll \dif \Pi_{C}(u) \le \frac{2\dif u}{\log 2u},
\] 
which of course holds in our case as well. We will employ the following inequality, which follows from an equality of Kolmogorov for sums of independent random variables (see \cite[Chapter V]{Loeve} or \cite[Lemma 8, p.~17]{DiamondMontgomeryVorhauer}). Let $X_{j}$ be independent Bernoulli variables with parameter $q_{j}$, and suppose that $r_{j}$ are real numbers with $\abs{r_{j}}\le 1$ ($j=1,\dotsc, J$). Set $X=\sum_{j=1}^{J}r_{j}X_{j}$. If 
\begin{equation}
\label{eq: condition v}
	0\le v\le 2\sum_{j=1}^{J}q_{j}(1-q_{j}), 
\end{equation}
then
\begin{equation}
\label{eq: probabilistic inequality}
	P(X\geq E(X) + v) \le \exp\biggl(\frac{-v^{2}}{4\sum_{j=1}^{J}q_{j}(1-q_{j})}\biggr).
\end{equation}
\begin{proof}[Proof of Lemma \ref{lem: probabilistic lemma}] As previously indicated, we only show the bound \ref{Bound (C)}, and we will assume that the bound \ref{Bound (A)} holds\footnote{If $P(A^{c}) \le \eps$, we can bound the probability of an event $D$ as $P(D) = P(D | A)P(A) + P(D | A^{c})P(A^{c}) \le P(D | A)P(A) + \eps$.}. By the trivial estimates $S(y;t), S_{C}(y;t) \ll y/\log y$, we may assume that $y \ge C\sqrt{\log\tau_{k}}(\log\log\tau_{k})^{2}$ for some fixed but arbitrarily large constant $C>0$. We apply the inequality \eqref{eq: probabilistic inequality} with $r_{k} = \cos(t\log v_{k})$ and $v = \sqrt{y}(\log\tau_{k})^{1/4}$. Let $J$ be such that $v_{J} \le y<v_{J+1}$. Using that (provided that $q_{j}\le 1/2$)
\[
	\frac{1}{2}\Pi_{C}(v_{J}) \le \sum_{j=1}^{J}q_{j}(1-q_{j}) \le \Pi_{C}(v_{J}), 
\]
we see that \eqref{eq: condition v} holds, so that by \eqref{eq: probabilistic inequality},
\[
	P(\Re S(y;t) \ge E(\Re S(y;t)) + \sqrt{y}(\log\tau_{k})^{1/4}) \le \exp\biggl(-\frac{y\sqrt{\log\tau_{k}}}{4\Pi_{C}(y)}\biggr) \le \exp\biggl(-\frac{1}{8}\log y\sqrt{\log\tau_{k}}\biggr).
\]
The fact that in the range 
\[
	\abs{t-\tau_{k}} \le \exp\biggl(\frac{c''}{2}\biggl(\frac{\log x_{k}}{\log\log x_{k}}\biggr)^{1/6}\biggr) \quad \mbox{and} \quad y \ge C\sqrt{\log\tau_{k}}(\log\log\tau_{k})^{2}
\]
$E(S(y;t))$ is close to $S_{C}(y;t)$ can be proven in exactly the same way as in \cite[pp.~21--22]{DiamondMontgomeryVorhauer}, although we need a different choice for the sequence $v_{j}$. The reader may check that the choice $v_{j} = (\log j)^{1/4}$ ($j\ge j_{0}$) is adequate for obtaining $\abs{E(S(y;t)) - S_{C}(y;t)} \le \sqrt{y}$ in the given ranges (and also works fine for the proof of the bounds \ref{Bound (A)} and \ref{Bound (B)}). Hence,
\[
	P\bigl(\Re S(y;t) \ge \Re S_{C}(y;t) + 2\sqrt{y}(\log\tau_{k})^{1/4}\bigr) \le \exp\biggl(-\frac{1}{8}\log y\sqrt{\log\tau_{k}}\biggr).
\]
Applying the same argument for $r_{j} = -\cos(t\log v_{j})$ and $r_{j} = \pm\sin(t\log v_{j})$ gives that 
\[
	P\bigl(\,\abs{S(y;t) - S_{C}(y;t)} \ge 4\sqrt{y}(\log\tau_{k})^{1/4}\bigr) \le 4\exp\biggl(-\frac{1}{8}\log y\sqrt{\log\tau_{k}}\biggr).
\] 
Let $C_{mk}$ denote the event $\abs{S(m;n) - S_{C}(m;n)} \ge 4\sqrt{m}(\log\tau_{k})^{1/4}$ for some integer $n$ in the range 
\[
	\abs{n-\tau_{k}} \le \exp\biggl(\frac{c''}{2}\biggl(\frac{\log x_{k}}{\log\log x_{k}}\biggr)^{1/6}\biggr).
\]
Using the relation \eqref{eq: xk} between $\tau_{k}$ and $x_{k}$, we get for some constant $c'''$
\begin{align*}
	P(C_{mk}) 	&\ll \sum_{n}\exp\biggl(-\frac{1}{8}\log m\sqrt{\log\tau_{k}}\biggr) \ll \exp\biggl(-\frac{1}{8}\log m\sqrt{\log\tau_{k}} + c'''\biggl(\frac{\log\tau_{k}}{\log\log\tau_{k}}\biggr)^{1/3}\biggr) \\
				&\le \exp\biggl(-\frac{1}{16}\log m\sqrt{\log\tau_{k}}\biggr),
\end{align*}
provided that $k$ is sufficiently large, but otherwise independent of $m$. Thus, by the rapid growth of the sequence $\tau_{k}$ (namely, property \ref{Property (a)} from Section \ref{sec: Setup and overview}),
\begin{align*}
	\sum_{k\ge K}\sum_{m\ge M}P(C_{mk}) 	&\ll \sum_{k\ge K}\sum_{m \ge M}\exp\biggl(-\frac{1}{16}\log m \cdot 5^{\frac{k}{2}}\sqrt{\log\tau_{0}}\biggr) \\
									&\le \sum_{k\ge K}\exp\biggl( -\biggl(\frac{5^{\frac{k}{2}}}{16}\sqrt{\log\tau_{0}} - 1\biggr)\log(M-1)\biggr) < \infty.
\end{align*}
We conclude that 
\[
	\forall\eps>0 \, \exists M_{\eps}, K_{\eps} \in \N\colon P\biggl(\bigcup_{m\ge M_{\eps}}\bigcup_{k\ge K_{\eps}}C_{mk}\biggr) < \eps.
\]
Consider now an event in the complement $\omega \in \bigcap_{m\ge M_{\eps}}\bigcap_{k\ge K_{\eps}}C_{mk}^{c}$. Then the bound \ref{Bound (C)} holds for integral $y$ and $t$. One then uses that $\abs{y_{1}-y_{2}} \le 1 \implies S(y_{1};t)  = S(y_{2};t) + O(\sqrt{y_{1}})$ (since $\pi(y) = \Pi_{C}(y) + O(\sqrt{y})$) and that 
\[
	S(y;t_{1}) = y^{\I(t_{2}-t_{1})}S(y;t_{2})  - \I(t_{2}-t_{1})\int_{1}^{y}S(u;t_{2})u^{\I(t_{2}-t_{1})-1}\dif u
\]
(and similarly for $S_{C}$) to see that the bound \ref{Bound (C)} also holds for non-integral $y$ and $t$.
\end{proof}

We now fix an event $\omega_{0}$ with corresponding prime number system $\mathcal{P}_{0}=\MP(\omega_{0})$ for which the bounds \ref{Bound (A)}-\ref{Bound (C)} hold. Denote the zeta function of this prime number system by $\zeta_{0}$. The bounds \ref{Bound (A)}-\ref{Bound (C)} imply the following bounds for $\zeta_0$.
\begin{lemma}
\label{lem: bounds zeta/zeta_{C}}
The function $\log \zeta_{0}(s) - \log \zeta_{C}(s)$ admits an analytic continuation to $\sigma>1/2$. Uniformly for $\sigma\ge 1/2 +\eps$ we have 
\[
	\log \zeta_{0}(s) = \log \zeta_{C}(s) + O_{\eps}(\sqrt{\log(\,\abs{t}+2)}),
\]
while for $k$ sufficiently large,
\begin{align*}
	\log\zeta_{0}(s) 	&= \log\zeta_{C}(s) + O_{\eps}((\log\tau_{k})^{1/4}), \\
	(\log\zeta_{0}(s))'	&= (\log\zeta_{C}(s))' + O_{\eps}((\log\tau_{k})^{1/4})
\end{align*}
uniformly in the range 
\[	
	\sigma\ge 1/2 + \eps, \quad \abs{t-\tau_{k}} \le \exp\biggl(\frac{c''}{2}\biggl(\frac{\log x_{k}}{\log\log x_{k}}\biggr)^{1/6}\biggr).
\]
\end{lemma}
One can check that the above bounds are strong enough so that one can repeat the proof of Theorem \ref{th: optimality HL} with $\zeta_{0}$ instead of $\zeta_{C}$ along the same contour, except for the estimation of the contribution from the saddle points. For the argument in Subsection \ref{subsec: The contribution from the saddle points} to go through, we would also require that on the paths of steepest descent  
\begin{equation}
\label{eq: Im zeta discrete system}
d(\Im(\log\zeta(s) - \log\zeta_{C}(s)), 2\pi\Z) < \pi/20,
\end{equation}
 say, but it is unclear whether this holds true for $\zeta=\zeta_0$. We will thus modify the prime number system $\mathcal{P}_{0}$ by adding a finite number of primes such that this bound holds, at least infinitely often for subsequences of $(\tau_{2k})_{k}$ and $(\tau_{2k+1})_{k}$.

Set 
\[
	S_{m}\coloneqq \biggl[ m\frac{\pi}{80} - \frac{\pi}{160}, m\frac{\pi}{80} + \frac{\pi}{160} \biggr) + 2\pi\Z, \quad \text{for } m=0, 1, \dotsc, 159.
\]
By the pigeonhole principle, there is a number $m$ (resp.\ $l$) such that for infinitely many even $k$ (resp.\ odd $k$), 
\[
	\Im(\log\zeta_{0}(1+\I\tau_{k}) - \log\zeta_{C}(1+\I\tau_{k})) \in S_{m} \quad (\text{resp. } S_{l}).
\]
Suppose without loss of generality that $l\le m$. We will add $m$ times the prime $p$ to $\mathcal{P}_{0}$, where $p$ is a well chosen number around $80/\pi$. This changes $\log\zeta_{0}(s)$ by $-m\log(1-p^{-s})$. This additional term and its derivative are $O(1)$, so Lemma \ref{lem: bounds zeta/zeta_{C}} still holds for the new zeta function. 

In $1+\I\tau_{k}$, the imaginary part of $\log \zeta_0$ changes by
\[
	-m\arg(1-p^{-1-\I\tau_{k}}) = -m\arctan\biggl(\frac{p^{-1}\sin(\tau_{k}\log p)}{1-p^{-1}\cos(\tau_{k}\log p)}\biggr).
\] 
Let $\alpha$ be a solution of 
\[
	\frac{\sin \alpha}{1-\frac{\pi}{80}\cos\alpha} = \frac{l}{m}, \quad 0\le \alpha\le \pi/2.
\]
We set 
\[
	p\coloneqq \frac{80}{\pi}\e^{\eps}, \quad \text{where } \eps = \sum_{k=0}^{\infty}\eps_{k}, \quad \eps_{k} \ll \frac{1}{\tau_{k}}.
\]
We define the numbers $\eps_{k}$ inductively: suppose $\eps_{0}, \eps_{1}, \dotsc, \eps_{k-1}$ are defined. Set $\eps_{k}\coloneqq \lambda_{k}/\tau_{k}$, with $\lambda_{k}\in [0, 2\pi)$ the unique number such that 
\[
	\tau_{k}\biggl( \log\frac{80}{\pi} + \eps_{0} + \eps_{1} + \dotsb + \eps_{k}\biggr) \in \frac{\pi}{2} + 2\pi\Z \quad (\text{ resp. } \in\alpha + 2\pi\Z), 
\]
for $k$ even (resp.\ odd). 

Suppose now that $k$ is even (the reasoning for odd $k$ is completely analogous). Using the rapid growth of the sequence $\tau_{k}$, that is, property \ref{Property (a)} from Section \ref{sec: Setup and overview}, we get
\begin{align*}
	\tau_{k}\log p 	&= \tau_{k}\biggl( \log\frac{80}{\pi} + \eps_{0} + \eps_{1} + \dotsb + \eps_{k}\biggr) + O\biggl( \sum_{n=1}^{\infty}\frac{\tau_{k}}{\tau_{k+n}}\biggr) \\
				&= \frac{\pi}{2} + 2\pi M_{k} + O(\tau_{k}^{-4}),
\end{align*} 
for some integer $M_{k}$. Then,
\[
	\sin(\tau_{k}\log p) = 1+O(\tau_{k}^{-8}), \quad \cos(\tau_{k}\log p) = O(\tau_{k}^{-4}), \quad p = \frac{80}{\pi} + O(\tau_{0}^{-1}),
\]
so that
\[
	\frac{p^{-1}\sin(\tau_{k}\log p)}{1-p^{-1}\cos(\tau_{k}\log p)} = \frac{\pi}{80} + O(\tau_{0}^{-1}).
\]
Since $\abs{\arctan u - u} < 3\abs{u}^{3}$ for $\abs{u}<1$, we have (for $\tau_{0}$ sufficiently large) for every even $k$
\[
	\abs{\Im(-m\log(1-p^{-1-\I\tau_{k}})) + m\frac{\pi}{80}} < 6m\biggl(\frac{\pi}{80}\biggr)^{3} < \frac{\pi}{40},
\]
and similarly for every odd $k$
\[
	\abs{\Im(-m\log(1-p^{-1-\I\tau_{k}})) + l\frac{\pi}{80}} < \frac{\pi}{40}.
\]

Set $F(s)\coloneqq \log\zeta_{0}(s) - \log\zeta_{C}(s) - m\log(1-p^{-s})$. Then, for an infinite number of even $k$ and an infinite number of odd $k$, $d(\Im F(1+\I\tau_{k}), 2\pi\Z) < 5\pi/160$. To see that such a (slightly weaker) bound also holds on the corresponding paths of steepest descent, write
\[
	F(s) = F(1+\I\tau_{k}) + \int_{1+\I\tau_{k}}^{s}F'(z)\dif z.
\]
For $s$ on such a path we have
\[
	\abs{s-1-\I\tau_{k}} \ll \frac{\log\log\tau_{k}}{\log\tau_{k}} + \frac{(\log x_{k})^{1/3}(\log\log x_{k})^{2/3}}{\log\tau_{k}},
\] 
and using the bound on the derivatives from Lemma \ref{lem: bounds zeta/zeta_{C}}, we get
\[
	\int_{1+\I\tau_{k}}^{s}F'(z)\dif z \ll \frac{(\log\tau_{k})^{2/3}(\log\log\tau_{k})^{1/3}}{\log\tau_{k}}(\log\tau_{k})^{1/4} \ll \frac{(\log\log\tau_{k})^{1/3}}{(\log\tau_{k})^{1/12}},
\]
so that for an infinite number of even $k$ (resp.\ odd $k$), we have $d(\Im F(s), 2\pi\Z) < \pi/20$ for $s$ on the corresponding paths of steepest descent. 

Therefore, the bound \eqref{eq: Im zeta discrete system} holds for $\zeta(s)=\zeta_{0}(s)(1-p^{-s})^{-m}$, the zeta function corresponding to the number system $\mathcal{P}$ obtained by adding $m$ times the prime $p$ to $\mathcal{P}_{0}$. This allows one to estimate the contribution from the saddle points as in Subsection \ref{subsec: The contribution from the saddle points}, and hence to deduce that \eqref{eq: discrete integer approximation} holds for the counting function of the set of generalized integers associated to $\mathcal{P}$. The prime counting function of this number system obviously satisfies \eqref{eq: discrete prime approximation} as well. This concludes the proof of Theorem \ref{th: discrete optimality HL}.

\appendix

\section{Improving the constant in Theorem \ref{th: HilberdinkLapidus}}
\label{appendix}

We indicate here how Balazard's method \cite{Balazard1999} yields an improvement for the value of $c$ in Theorem \ref{th: HilberdinkLapidus} over the value given by Hilberdink and Lapidus \cite{HilberdinkLapidus2006}.
\begin{theorem} \label{th: ImprovedHL} Suppose that the generalized Riemann prime counting function satisfies
\begin{equation*}
 \Pi(x) = \int_{1}^{x}\frac{1-u^{-1}}{\log u}\dif u + O(x^{\theta}),
\end{equation*}
for some $0 \leq \theta < 1$. Then, there is $\rho > 0$ such that the generalized integer counting function satisfies, for each $c < \sqrt{2(1-\theta)}$,
\begin{equation*}
	N(x) = \rho x + O\left(x\exp\left(-c\sqrt{\log x \log \log x}\right)\right).
\end{equation*}
\end{theorem}
Naturally, when $\theta \geq 1/2$, Theorem \ref{th: ImprovedHL} implies Theorem \ref{th: HilberdinkLapidus} and improves the values of the constant $c$. In its proof, we shall make extensive use of the operational calculus for multiplicative convolution of measures for which we can refer the reader to \cite[Chapter 2]{DiamondZhangbook}. 

\begin{proof} We define $\mathrm{d}P$ as the ``natural" continuous prime measure and $\mathrm{d}E$ as the deviation between $\mathrm{d}\Pi$ and $\mathrm{d}P$, that is,
\begin{equation*}
 \mathrm{d}P = \frac{1-u^{-1}}{\log u} \mathrm{d}u, \qquad \mathrm{d}E = \mathrm{d}\Pi - \mathrm{d}P.
\end{equation*}
By hypothesis, $E(x):= \int^{x}_{1} \mathrm{d}E = O(x^{\theta})$. We recuperate $N$ as
\begin{align*}
 N(x) & = \int^{x}_{1^{-}} \expm (\mathrm{d}\Pi) = \int^{x}_{1^{-}} \expm (\mathrm{d}P) \ast \expm (\mathrm{d}E) =  x \int^{x}_{1^{-}} u^{-1} \expm(\mathrm{d}E(u))\\
& = x \int^{x}_{1^{-}} \expm\left(\frac{\mathrm{d}E(u)}{u}\right) = x \sum_{n = 0}^{\infty} \frac{1}{n!} \int^{x}_{1^{-}} \left(\frac{\mathrm{d}E(u)}{u}\right)\nconv,
\end{align*}
where we have used the identity $\expm (\mathrm{d}P) = \delta_{1} + \mathrm{d}u$, where $\delta_{1}$ denotes the Dirac measure concentrated at $1$. 
We are therefore led to the integrals 
\begin{equation*}
 I_{n} = \int^{x}_{1^{-}} \left(\frac{\mathrm{d}E(u)}{u}\right)\nconv.
\end{equation*}
The first two cases are trivial, $I_{0} = 1$ and $I_{1} = b + O(x^{-(1-\theta)})$, where $b = \int^{\infty}_{1}u^{-2}E(u)\mathrm{d}u$.\par
For larger $n$, we claim that there is an absolute constant $A$ for which
\begin{equation} \label{eq: higherorderconv}
 I_{n} = b^{n} + O\left(\frac{A^{n} (\log \log x)^{n-1}}{x^{(1 -\theta)/n}}\right),
\end{equation}
and where the $O$-constant is independent of $n$. \par
We proceed to show the claim via induction. Employing the Dirichlet hyperbola method we decompose $I_{n+1}$ as $S_{1} + S_{2} - S_{3}$, where $y$ shall be taken as the free variable to be optimized at our will. The first term is
\begin{align*}
 S_{1}  &:= \int^{y}_{1} \frac{\mathrm{d}E(u)}{u} \int^{x/u}_{1} \left(\frac{\mathrm{d}E(v)}{v}\right)\nconv = \int^{y}_{1} \left(b^{n} + O\left(\frac{A^{n}(\log \log x)^{n-1}u^{(1-\theta)/n}}{x^{(1-\theta)/n}}\right)\right)\frac{\mathrm{d}E(u)}{u} \\
& = b^{n+1} + O\left(\frac{b^{n}}{y^{1 - \theta}}\right) + O\left(\frac{A^{n}(\log \log x)^{n-1}}{x^{(1-\theta)/n}}\int^{y}_{1} u^{-1+(1-\theta)/n}|\mathrm{d}E(u)|\right).
\end{align*}
Subsequently, as $|\mathrm{d}E| \leq \mathrm{d}\Pi + \mathrm{d}P = 2 \mathrm{d}P + \mathrm{d}E$, the integral in the final error term can, for large enough $x$, be estimated by
\begin{equation*}
\int^{y}_{1} u^{-1+(1-\theta)/n}|\mathrm{d}E(u)| \leq  y^{(1-\theta)/n}\int^{x}_{1} u^{-1}(2 \mathrm{d}P(u) + \mathrm{d}E(u)) \leq 3 y^{(1-\theta)/n} \log \log x.
\end{equation*}
The second summand is estimated using the representation
\begin{align*}
 S_{2} & := \int^{x/y}_{1} \left(\frac{\mathrm{d}E(v)}{v}\right)\nconv \int^{x/v}_{1} \frac{\mathrm{d}E(u)}{u}  = \int^{x/y}_{1} \left(b + O\left(\frac{v^{1 - \theta}}{x^{1 - \theta}}\right)\right) \left(\frac{\mathrm{d}E(v)}{v}\right)\nconv\\
& = b^{n+1} + O\left(\frac{A^{n} (\log \log x)^{n}y^{(1-\theta)/n}}{x^{(1-\theta)/n}}\right) + O\left(\frac{1}{y^{1-\theta}}\int^{x/y}_{1} \left(\frac{|\mathrm{d}E(v)|}{v}\right)\nconv\right)
\end{align*}
Considering the integral involving the $n$-folded convolution of $|\mathrm{d}E|$, we obtain, for $x$ large enough, 
\begin{align*}
 \int^{x/y}_{1} \left(\frac{|\mathrm{d}E(v)|}{v}\right)\nconv & \leq \left( \int^{x/y}_{1}   \frac{|\mathrm{d}E(v)|}{v} \right)^{n} \leq \left( \int^{x/y}_{1}   \frac{2 \mathrm{d}P(v) + \mathrm{d}E(v)}{v} \right)^{n} \\
& \leq \left(3\log \log x\right)^{n}.
\end{align*}
Finally, $S_{3}$ is bounded using
\begin{equation*}
 S_{3} := \int^{y}_{1} \frac{\mathrm{d}E(u)}{u} \int^{x/y}_{1} \left(\frac{\mathrm{d}E(v)}{v}\right)\nconv = b^{n+1} + O\left(\frac{b^{n}}{y^{1-\theta}}\right) + O\left(\frac{A^{n}(\log \log x)^{n-1}y^{(1-\theta)/n}}{x^{(1-\theta)/n}}\right).
\end{equation*}
Upon choosing $y = x^{1/(n+1)}$ and $A$ sufficiently large such that the extra factor $A$ can absorb all the absolute $O$-constants, the claim \eqref{eq: higherorderconv} follows.\par
Now it remains to analyze 
\begin{equation} \label{eq: sumIn}
 \sum_{n = 0}^{\infty} \frac{I_{n}}{n!} = \e^{b} + O\left(\sum_{n= 1}^{\infty}\frac{A^{n} (\log \log x)^{n-1}}{n!x^{(1-\theta)/n}} \right).
\end{equation}
We now select the $n$, say $n_{\max}$, for which $2^{n}$ times the $n$-th term in the above series reaches its maximum. A few standard calculations allows one to find an approximation for $n_{\max}$,
\begin{equation*}
 n_{\max} = \sqrt{2(1-\theta)}\sqrt{\frac{\log x}{\log \log x}}\left(1 + O\left(\frac{\log \log \log x}{\log \log x}\right)\right),
\end{equation*}
and after inserting this into \eqref{eq: sumIn}, one obtains
\begin{equation*} 
N(x) = \e^{b}x + O\left(x\exp\left(- \sqrt{2(1-\theta)}\sqrt{\log x \log \log x} + O\bigg(\frac{(\log x)^{1/2} \log \log \log x}{(\log \log x)^{1/2}}\bigg)\right)\right). \qedhere
\end{equation*}
\end{proof}


\end{document}